\PassOptionsToPackage{dvipsnames}{xcolor} %%xcoloroptioncrashsolver

\documentclass[10pt,a4paper]{amsart}
\usepackage[a4paper, left=1in, right=1in, top=1.1in, bottom=1.1in]{geometry}

\usepackage[english]{babel}
\usepackage[utf8]{inputenc}
\usepackage{amsmath}
\usepackage{graphicx}
\usepackage{amssymb}
\usepackage{stmaryrd}
\usepackage{amsthm}
\usepackage{tikz-cd}
\usepackage{mathrsfs}
\usepackage{csquotes}
\usepackage[colorinlistoftodos]{todonotes}
\usepackage{yfonts}
\usepackage{ dsfont }
\usepackage{xcolor}
\usepackage{multicol}
\usepackage{indentfirst}
\usepackage[hyphens]{url}
\usepackage{hyperref}
\usepackage[hyphenbreaks]{breakurl}

 \usepackage[
    backend=biber,
    style=alphabetic, maxbibnames=99, url=false
  ]{biblatex}
\usepackage{xurl}
\hypersetup{
    colorlinks,
    linkcolor={red!70!black},
    citecolor={blue!70!black},
    urlcolor={violet}
}
\hypersetup{breaklinks=true}

\usepackage[inline]{enumitem}   
\makeatletter

\newcommand{\inlineitem}[1][]{%
\ifnum\enit@type=\tw@
    {\descriptionlabel{#1}}
  \hspace{\labelsep}%
\else
  \ifnum\enit@type=\z@
       \refstepcounter{\@listctr}\fi
    \quad\@itemlabel\hspace{\labelsep}%
\fi}
\makeatother
\allowdisplaybreaks

 \usepackage[
    backend=biber,
    style=alphabetic, maxbibnames=99, url=false
  ]{biblatex}
\usepackage[capitalise]{cleveref}

\title[Line bundles on Contractions of $\overline{\rm{M}}_{0,n}$ via Coinvariant Divisors]{Line bundles on Contractions of $\overline{\rm{M}}_{0,n}$\\ via Coinvariant Divisors}

\author{Daebeom Choi}
\address{Department of Mathematics\\
    University of Pennsylvania\\
    Philadelphia, PA 19104-6395}
\email{dbchoi@sas.upenn.edu}

\date{\today}
\theoremstyle{definition}
\newtheorem{thm}{Theorem}[section]

\newtheorem{lem}[thm]{Lemma}

\newtheorem{prop}[thm]{Proposition}
\newtheorem{qes}[thm]{Question}

\newtheorem{defn}[thm]{Definition}
\newtheorem{eg}[thm]{Example}

\newtheorem{conj}[thm]{Conjecture}

\newtheorem{cor}[thm]{Corollary}

\newtheorem{rmk}[thm]{Remark}

\newcommand{\R}{\mathbb{R}}
\newcommand{\C}{\mathbb{C}}
\newcommand{\Z}{\mathbb{Z}}
\newcommand{\N}{\mathbb{N}}
\newcommand{\Q}{\mathbb{Q}}
\newcommand{\rank}{\text{rank}}

\newcommand{\M}[1]{\overline{\rm{M}}_{0,#1}}
\newcommand{\NE}[1]{\overline{\text{NE}}_1(#1)}

\DeclareSymbolFont{yhlargesymbols}{OMX}{yhex}{m}{n}
\DeclareMathAccent{\widetriangle}{\mathord}{yhlargesymbols}{"E6} %% widetriangle

\begin{document}

\begin{abstract}
    Using representations of affine Lie algebras, we describe line bundles on a broad class of contractions of $\M{n}$, the moduli space of stable $n$-pointed rational curves, and show a variant of the cone and contraction theorem for these morphisms. These include the celebrated constructions of Kapranov, Keel, and Knudsen. Our main result suggests that while many so-called F-curves are not $K_X$-negative, they exhibit behavior similar to $K_X$-negative curves. This reveals a distinguished property of Knudsen's construction $f_{\text{Knu}}:\M{n}\to \M{n-1}\times_{\M{n-2}}\M{n-1}$, allowing for the classification of all possible factorizations of $f_{\text{Knu}}$, as well as further applications.
\end{abstract}

\maketitle

\section{Introduction}\label{sec:intro}
The Deligne-Mumford moduli space $\overline{\rm{M}}_{g,n}$ of stable $n$-pointed curves of genus $g$ is a well-studied object in algebraic geometry, but its geometry remains a considerable challenge. Even in the case of $g=0$, where $\M{n}$ is smooth and rational, the problem remains difficult.

To study $\M{n}$, it is useful to construct it from a simpler variety $Y$ and to study the associated contraction $f_Y : \M{n} \to Y$. Three important examples of this are constructions $f_{\text{Kap}}:\M{n}\to \mathbb{P}^{n-3}$ considered by Kapranov in \cite{Ka92,Ka93}, $f_{\text{Knu}} : \M{n} \to \M{n-1} \times_{\M{n-2}} \M{n-1}$, defined by Knudsen in \cite{Knu83} and $f_{\text{Keel}}:\M{n}\to \M{n-1}\times \M{4}$, described by Keel in \cite{Ke92}. There exists a common generalization, $f_{S,T}:\M{n}\to \M{S}\times_{\M{S\cap T}}\M{T}$, of the last two. These are reviewed in \cref{subsec:cont}.

Here, we obtain new, generalizable results about such contractions of $\M{n}$ through the use of first Chern classes of vector bundles of coinvariants, derived from representations of affine Lie algebras, i.e. coinvariant divisors (cf. \cref{subsec:voa} and \cref{rmk:cdiv}). We focus on Mori theoretic aspects of $\M{n}$. 

Mori theory classifies morphisms $f:X\to Y$ through its relative cone of curves $\NE{f}$. In particular, if $\NE{f}$ is generated by $K_X$-negative curves, then $\text{Pic}(Y)$ is characterized by the subgroup of $\text{Pic}(X)$ that intersect trivially with $\NE{f}$ (cf. \cite[Lemma 3.2.5]{KMM87}). The main theorem of this paper proves that this holds for $f_{\text{Kap}}$, $f_{\text{Keel}}$, and $f_{\text{Knu}}$, and for the generalizations $f_{S,T}$.

\begin{thm}\label{thm:charcont}
    Let $f=f_{S,T}$ or $f_{\text{Kap}}$, and let $F$ be the set of F-curves contracted by them. Then a line bundle on $\M{n}$ is a pullback of a line bundle along $f$ if and only if its intersection with $F$ is trivial.
\end{thm}
The notion of F-curves and the explicit set of F-curves contracted by these morphisms are described in \cref{subsec:fconj} and \cref{subsec:cont}. Note also that \cref{thm:charcont} holds in arbitrary characteristic.

Note that some instances of \cref{thm:charcont} already appear in the literature. The statements for $f_{\text{Kap}}$ and $f_{\text{Keel}}$ are special cases of a theorem by V. Alexeev (see \cite[Proposition 4.6]{Fak12}). Moreover, in characteristic zero, the case of $f_{\text{Knu}}$ can be proved using the contraction theorem. However, our method offers certain advantages over existing approaches, which we will discuss in \cref{subsec:nec}. 
\cref{thm:charcont} follows from a combination of \cref{thm:charKap}, \cref{thm:charKnu}, and \cref{thm:chargen}. \cref{thm:charcont} is related to the semiampleness of certain line bundles, as explained in \cref{prop:nefsem}.

As a first application, it yields a concrete description of $\text{Pic}\left(\M{S} \times_{\M{S \cap T}} \M{T} \right)$, revealing a distinguished and non-obvious property of the double fiber product (cf. \cref{eg:triple}). As another application, \cref{thm:charcont} provides new insight into the relative cone of curves, as shown in \cref{thm:Knuind} for the case of $f_{\text{Knu}}$. In particular, this facilitates the classification of those contractions of $\M{n}$ whose relative cone of curves is contained in $\NE{f_{\text{Knu}}}$, as presented in \cref{thm:Knumori} and \cref{thm:Knugeo}, thereby refining Knudsen’s construction. A key feature of $f_{\text{Knu}}$ is that $\NE{f_{\text{Knu}}}$ is simplicial (\cref{prop:kkne}) and supports a full exact sequence (\cref{thm:charKnu}). These properties underpin the arguments in \cref{subsec:Mori} and also support the inductive approach of \cref{subsec:app} (cf. \cref{rmk:Knu}).

A detailed explanation of the connection demands additional context and background on Mori theory. A crucial tool in Mori theory is the closed cone of curves. The F-conjecture asserts that $\NE{\overline{\rm{M}}_{g,n}}$ is the convex hull of finitely many special curves, called F-curves. We refer to \cref{subsec:fconj}  for a detailed description of the F-conjecture. Some progress has been made on the F-conjecture \cite{KM13, GKM02, GF03, Gib09, Fe15, MS19, Fe20}. To illustrate the importance of $\M{n}$, we note that by \cite[Theorem 0.7]{GKM02} the F-conjecture for $\M{n}$ would imply its validity for $\overline{\rm{M}}_{g,m}$ for $g+m\le n$.

The cone of curves encodes information about  contractions of a space. The rigidity theorem asserts that any contraction $f$ is characterized by its relative closed cone of curves $\NE{f}$. Thus, if the F-conjecture is true, any contraction of $\M{n}$ is determined by the F-curves it contracts. Furthermore, by the contraction theorem, for a projective smooth variety $X$, and an extremal face $F$ of $\NE{X}$ consisting of $K_X$-negative curves, there exists a contraction $f: X \to Y$ onto a projective variety $Y$, where $\NE{f} = F$. Moreover, the image of the inclusion $f^\ast:\text{Pic}(Y)\to \text{Pic}(X)$ is $\NE{f}^\perp$, i.e., the line bundles intersecting $\NE{f}$ trivially by \cite[Lemma 3.2.5]{KMM87}. However, this description is not guaranteed if $F$ contains $K_X$-nonnegative curves, as discussed in \cite[Remark 3.2.6]{KMM87}.

%%In this setting, \cref{thm:charcont} is applicable in situations where \cite[Lemma 3.2.5]{KMM87} fails to apply.

When applied to this setting, \cref{thm:charcont} appears to capture certain phenomena that may go beyond the scope of \cite[Lemma 3.2.5]{KMM87}. First, we only need F-curves, rather than the entire relative cone of curves to obtain the conclusion. In particular, \cref{thm:charcont} may be considered as evidence for the F-conjecture. Secondly, as previously noted, \cref{thm:charcont} reveals that $f_{\text{Kap}}$, $f_{\text{Keel}}$, $f_{\text{Knu}}$ and the more general $f_{S,T}$ satisfy this description even though their cones of contracted curves contain $K_{\M{n}}$-non-negative curves. Notably, among the F-curves, only those of the form $F(I,J,K,L)$ with $|I|=|J|=|K|=1$ are $K_{\M{n}}$-negative.

Moreover, \cref{thm:Knumori} shows the existence of contractions corresponding to every extremal face of $\NE{f_{\text{Knu}}}$, although there are $K_{\M{n}}$-trivial curves in $\NE{f_{\text{Knu}}}$. Hence, \cref{thm:charcont} and \cref{thm:Knumori} suggest that even though many curves in $\M{n}$ are not $K_{\M{n}}$-negative, they behave as if they were. This observation aligns with the F-conjecture. According to the cone theorem, the $K_X$-negative part of $\NE{X}$ is locally finitely generated, with extremal rays spanned by rational curves. The F-conjecture proposes that the same holds true for $\M{n}$, even though most of its curves are not $K_X$-negative.

The main tool in the proof of \cref{thm:charcont} is the use of coinvariant divisors. As will be explained in \cref{subsec:voa}, coinvariant divisors are a special class of line bundles on $\overline{\rm{M}}_{g,n}$ derived from modules over affine Lie algebras. These divisors behave functorially with respect to tautological maps, for instance satisfying propagation of vacua \cite[Proposition 2.2.3]{TUY89} (cf. \cref{thm:propvac}) and the factorization theorem \cite[Theorem 6.2.6]{TUY89} (cf. \cref{thm:factor}), making calculations easier. Furthermore, $\mathfrak{sl}_2$, level $1$ coinvariant divisors form a basis of $\text{Pic}\left(\M{n}\right)_\Q$ by \cite[Lemma 2.5]{Fak12} (cf. \cref{thm:basis}). Proving \cref{thm:charcont} involves solving linear equations in $\text{Pic}\left(\M{n}\right)_\Q$. The fact that the basis of $\text{Pic}\left(\overline{\rm{M}}_{0,n}\right)_\Q$ consists of coinvariant divisors makes such linear equations easier to solve since their restrictions to boundary strata can be easily calculated by applying propagation of vacua and factorization. There are other bases, such as \cite[Lemma 2]{GF03}, but proving \cref{thm:charcont} using them is a very challenging task since such bases are not stable under pullback along projection and clutching maps. Note that although the proof uses coinvariant divisors, which are currently only defined over $\C$, the theorem also holds over any base field since it is solely a statement on the Chow ring of $\M{n}$. See Remark \ref{rmk:charp} for details.

Now, we list some further consequences of \cref{thm:charcont}. One interesting corollary is the exact sequence \cref{thm:charKnu}, which reveals a new inductive structure on line bundles on $\overline{\rm{M}}_{0,n}$. For example, it can be used to investigate generators of the Picard group (\cref{thm:CDint}). Moreover, we have the following.

\begin{cor}\label{thm:extray}
    The $\psi$-classes, along with any of their pullbacks by projection maps, span extremal rays of the nef cone. Moreover, such divisors span extremal rays of the F-nef cone.
\end{cor}

\begin{cor}\label{thm:Knuind}
    $F_{\text{Knu}}$, described in \cref{subsec:cont}, is linearly independent, and $\NE{f_{\text{Knu}}}$ is the simplicial cone generated by $F_{\text{Knu}}$. In particular, the F-conjecture holds for the subspace generated by $F_{\text{Knu}}$.
\end{cor}

Note that a direct proof of \cref{thm:Knuind} is difficult, as it involves handling limits of sums of effective curves, which are generally hard to control. Instead, we take a dual approach by studying the cone of nef line bundles. This perspective allows us to construct many elements using coinvariant divisors, as demonstrated in \cref{cor:knusln}.

\cref{thm:extray} and \cref{thm:Knuind} provide an improvement of similar results about extremal rays and linear independence, such as \cite[Theorem B, Proposition 4.1]{AGSS12} and \cite[Theorem 2.1, Theorem 6.1]{AGS14}. First, those results are about symmetrized F-curves and divisors, while symmetry is not assumed here. Second, the divisors in \cref{thm:extray} are tautological.
\cref{thm:extray} and \cref{thm:Knuind} follow from \cref{cor:Knuind}, \cref{cor:psiext}, and \cref{cor:genext}.

It's worth noting that this method of proof of \cref{thm:charcont} is not limited to $f_{S,T}$'s. This approach has been successfully applied to two distinct constructions, Kapranov's and $f_{S,T}$'s, suggesting its potential for characterizing line bundles on more general contractions of $\M{n}$. This versatility can be attributed to the simple nature of the $\mathfrak{sl}_2$ level 1 coinvariant divisors described earlier.

Note that we can prove some cases of main theorems using classical methods instead of coinvariant divisors. However, these methods cannot prove the full statement of \cref{thm:charcont} or \cref{thm:Knuind}. Moreover, coinvariant divisors offer a new algebraic perspective even when classical methods are applicable. We will discuss this in detail in \cref{subsec:nec}.

\cref{thm:Knuind} demonstrates that the relative closed cone of curves $\NE{f_{\text{Knu}}}$ is just the simplicial cone on $F_{\text{Knu}}$. Consequently, it is natural to ask whether describing the Mori theory of $\M{n}$ for $\NE{f_{\text{Knu}}}$, specifically associating contractions to each face, is applicable. Indeed, this is feasible by employing the coinvariant divisors.

\begin{thm}\label{thm:Knumori}
    We can associate:
    \begin{description}
        \item[Object] For any subset $A\subseteq F_{\text{Knu}}$, a projective variety $\overline{\mathrm{M}}_{0,n}^{\text{Knu}}(A)$.
        \item[Morphism] For any two subsets $A,B\subseteq F_{\text{Knu}}$ such that $A\subseteq B$, a birational contraction 
        \[f_{A,B}:\overline{\mathrm{M}}_{0,n}^{\text{Knu}}(A)\to \overline{\mathrm{M}}_{0,n}^{\text{Knu}}(B).\]
    \end{description}
    These associations satisfy:
    \begin{enumerate}
        \item $\overline{\mathrm{M}}_{0,n}^{\text{Knu}}(\emptyset)=\overline{\mathrm{M}}_{0,n}$, $\overline{\mathrm{M}}_{0,n}^{\text{Knu}}(F_{\text{Knu}})=\M{n}\times_{\M{n-1}}\M{n}$.
        \item $f_{A,B}$ is transitive, i.e. $f_{B,C}\circ f_{A,B}=f_{A,C}$.
        \item $f_{\emptyset, F_{\text{Knu}}}=f_{\text{Knu}}$.
        \item $\NE{f_{\emptyset, A}}=\NE{A}$, and $f_{\emptyset, A}$ is the unique contraction of $\overline{\mathrm{M}}_{0,n}$ with this property.
        \item $f_{\emptyset, A}^\ast\text{Pic}\left(\overline{\mathrm{M}}_{0,n}^{\text{Knu}}(A) \right)$ is the set of line bundles on $\overline{\mathrm{M}}_{0,n}$ which intersect with curves in $A$ trivially. In particular, $\text{Pic}\left(\overline{\mathrm{M}}_{0,n}^{\text{Knu}}(A)\right)$ is a free abelian group of rank $2^{n-1}-\binom{n}{2}-|A|-1$.
    \end{enumerate}
\end{thm}

Note that \cref{thm:Knumori} holds for any characteristic. \cref{thm:Knumori} will be explained in \cref{subsec:Mori}. Since any subset $A\subseteq F_{\text{Knu}}$ corresponds to a face of the cone $\NE{F_{\text{Knu}}}=\NE{f_{\text{Knu}}}$, \cref{thm:Knumori} completely classifies the contraction of $\M{n}$ whose relative closed cone of curves is contained in $\NE{f_{\text{Knu}}}$. There is a question, as posed for instance in \cite[Question 1.0.2]{Fe15}, about whether every nef line bundle on $\M{n}$ is semiample. This question is related to the disproved Mori dream space conjecture \cite{HK00, CT15}. \cref{thm:Knumori} provides supporting evidence for this conjecture, as clarified by \cref{prop:nefsem}.

The role of coinvariant divisors in the proof of \cref{thm:Knumori} is to provide base point free line bundles that contract the specific F-curves we are interested in. Owing to the global generation of coinvariant divisors, as established in \cite{Fak12} (cf. \cref{thm:glogen}), these divisors are base point free. Furthermore, by applying the factorization theorem (\cref{thm:factor}), we can construct such divisors to have the desired intersections with specified F-curves. Note that there are other ways to prove \cref{thm:Knumori} without using coinvariant divisors. However, these methods cannot supersede the use of coinvariant divisors. We will discuss this in \cref{subsec:nec}.

The main point of \cref{thm:Knumori} is to prove the projectivity of varieties $\overline{\mathrm{M}}_{0,n}^{\text{Knu}}(A)$. Indeed, it is possible to construct such varieties even without the coinvariant divisors, as we will do in \cref{subsec:Mori}. However, constructing an ample line bundle on $\overline{\mathrm{M}}_{0,n}^{\text{Knu}}(A)$ is challenging, as explained in \cref{rmk:proj}, \cref{eg:segre} and \cref{eg:hironaka}. Nevertheless, as mentioned in the previous paragraph, coinvariant divisors that are carefully constructed can provide an ample line bundle on $\overline{\mathrm{M}}_{0,n}^{\text{Knu}}(A)$. We explore this in more detail in \cref{subsec:Mori}, especially in \cref{eg:hironaka}.

As mentioned, there is another method for constructing $\overline{\mathrm{M}}_{0,n}^{\text{Knu}}(A)$, that involves gluing together open subvarieties of $\M{n}$ and $\M{n}\times_{\M{n-1}}\M{n}$. Utilizing this method of construction, we can establish several theorems concerning the regularity of $\overline{\mathrm{M}}_{0,n}^{\text{Knu}}(A)$.

\begin{prop}\label{thm:Knugeo}
    For any $A\subseteq F_{\text{Knu}}$,
    \begin{enumerate}
        \item $\overline{\mathrm{M}}_{0,n}^{\text{Knu}}(A)$ is a local complete intersection.
        \item $\overline{\mathrm{M}}_{0,n}^{\text{Knu}}(A)$ is smooth if and only if $A\subseteq F_{\text{Knu}}^{1}$. If $A\not\subseteq F_{\text{Knu}}^{1}$, then $\overline{\mathrm{M}}_{0,n}^{\text{Knu}}(A)$ is not even $\mathbb{Q}$-factorial.
        \item $f_{A, A\cup [C]}$ is a divisorial contraction if and only if $C\in F_{\text{Knu}}^{1}$. Otherwise, this is a small contraction.
    \end{enumerate}
\end{prop}

See \cref{subsec:Mori} for the definition of $F_{\text{Knu}}^{1}$. This theorem will be presented in \cref{prop:strknumori}. In particular, $\overline{\mathrm{M}}_{0,n}^{\text{Knu}}\left(F_{\text{Knu}}^{1}\right)$ is the minimal resolution of $\M{n-1}\times_{\M{n-2}}\M{n-1}$ among contractions of $\M{n}$. Despite these observations, several questions remain about $\overline{\mathrm{M}}_{0,n}^{\text{Knu}}(A)$. For instance, we seek for a modular description of these varieties. 

\subsection{Structure of the paper}\label{subsec:struc}
In \cref{sec:prelim}, we provide an overview of essential preliminaries for this paper. We introduce key concepts such as the F-conjecture, the sheaf of coinvariants, and coinvariant divisors. \cref{subsec:cont} delves into important contractions of $\M{n}$, including $f_{\text{Knu}}$, $f_{\text{Keel}}$, $f_{\text{Kap}}$, and the more general $f_{S,T}$. \cref{subsec:proof} is dedicated to proving \cref{thm:charcont}. In \cref{sec:cons}, we derive several corollaries of the main theorem. Specifically, \cref{subsec:Mori} deals with the Mori theory of $\M{n}$ with respect to $\NE{f_{\text{Knu}}}$ and includes proofs of \cref{thm:Knumori} and \cref{thm:Knugeo}. Finally, in \cref{sec:conclu}, we discuss the advantages of using coinvariant divisors in this work and explore some applications.

\subsection{Notation}\label{subsec:nota}
For any variety $X$, let $\text{A}_d(X)$ be the Chow group of $d$-cycles modulo numerical equivalence, and $\NE{X}$ be the closure of the cone of effective curves on $X$, both with $\Q$-coefficients. We denote $\text{A}_d(X,\Z)$ as the integral Chow group of $d$-cycles. If $f:X\to Y$ is a morphism, then let $\text{A}_1(f)$ be the kernel of $f_\ast:\text{A}_1(X)\to \text{A}_1(Y)$ and $\NE{f}:=\NE{X}\cap \text{A}_1(f)$. If $S$ is a subset of $\text{A}_1(X)$, then we will denote $\text{A}_1(S)$ as the $\Q$-vector subspace of $\text{A}_1(X)$ spanned by $S$, and $\NE{S}=\NE{X}\cap \text{A}_1(S)$. 

For $1 \le i \le n$, we define $\pi_i \colon \M{n} \to \M{n-1}$ to be the projection map that forgets the $i$th marked point. Note that $\pi_n \colon \M{n} \to \M{n-1}$ can be identified with the universal curve over $\M{n-1}$. For $1 \le i \le n-1$, we denote by $s_i \colon \M{n-1} \to \M{n}$ the $i$th section of the universal curve. If $S \subset \{1, \ldots, n\}$, we denote by $\pi_S \colon \M{n} \to \M{|S^c|}$ (resp. $\pi^S \colon \M{n} \to \M{|S|}$) the projection map that forgets the points indexed by $S$ (resp. the points not indexed by $S$). 

$[n]$ is the set $\left\{1,2,\cdots, n\right\}$. For any subset $I \subseteq [n]$ with $|I|, |I^c| \ge 2$, we denote by $\Delta_{0,I}$ the subvariety of $\M{n}$ consisting of stable curves with a node that separates the curve into two components, one containing $I$ and the other containing $I^c$. We define $\delta_{0,I}$ to be the divisor class corresponding to $\Delta_{0,I}$.

\section*{Acknowledgement}

The author would like to thank Angela Gibney for introducing the subject, her continued support, as well as for many helpful discussions. We are grateful to Brendan Hassett for explaining \cref{eg:segre}. We also appreciate Han-Bom Moon and Najmuddin Fakhruddin for their helpful comments on a preliminary version of the manuscript. Among all this useful feedback, Han-bom Moon provided additional examples of non-projective contractions of $\M{n}$ (cf. \cref{eg:segre}), and Najmuddin Fakhruddin provided an alternative proof of \cref{thm:charKap} and \cref{cor:charkeel} using classical methods (cf. \cref{subsec:nec}). Additionally, the author is indebted to Yeongjong Kim for communicating the proof of \cref{lem:linalg}.

\section{Preliminaries}\label{sec:prelim}

\subsection{F-curves and the F-conjecture}\label{subsec:fconj}
There is a natural stratification on $\overline{\rm{M}}_{g,n}$ defined as follows: $\rm{M}_{g,n}\subseteq\overline{\rm{M}}_{g,n}$ is an open subset, whose complement is the union of boundary divisors. The boundary divisors are images of clutching maps $\overline{\rm{M}}_{g_1,n_1}\times \overline{\rm{M}}_{g_2,n_2}\to \overline{\rm{M}}_{g,n}$ or  $\overline{\rm{M}}_{g-1,n+2}\to \overline{\rm{M}}_{g,n}$. Since the domain of clutching maps is again described in terms of moduli spaces of curves, this procedure can be iterated by repeatedly taking the open stratum of the domain that parametrizes smooth curves, and taking its complement. This defines a natural stratification on $\overline{\rm{M}}_{g,n}$, whose codimension $r$ stratum parametrizes stable curves with $r$ nodes. We refer to \cite[Section 2]{AC09} for the details of this stratification. The \textbf{boundary strata} are the closures of open strata of this stratification. Each irreducible component of the one-dimensional boundary stratum is called an \textbf{F-curve}. These strata are irreducible components of the closed subvariety of $\overline{\rm{M}}_{g,n}$ that parametrize stable curves with at least $3g-4+n$ nodes. There exist natural maps $\overline{\rm{M}}_{0,4}\to \overline{\rm{M}}_{g,n}$ and $\overline{\rm{M}}_{1,1}\to \overline{\rm{M}}_{g,n}$, as described in the next paragraph in the $g=0$ case, whose images are precisely the F-curves, each of which is isomorphic to $\mathbb{P}^1$.

\begin{conj}\label{conj:Fconj}(The F-conjecture)
    $\NE{\overline{\rm{M}}_{g,n}}$ is generated by the cycle class of F-curves. Equivalently, a line bundle on $\overline{\rm{M}}_{g,n}$ is nef if and only if its intersection number with each F-curve is nonnegative.
\end{conj}

The F-conjecture was first posed implicitly as a question in \cite[Question 1.1]{KM13} and then explicitly stated and investigated as a conjecture in \cite[Conjecture 0.2]{GKM02}.

As this paper primarily concerns $\M{n}$, we provide an explicit description of F-curves on $\M{n}$. For the general case, we refer to \cite[Section 2]{GKM02}. Let $I, J, K, L$ be a partition of $\left\{1,\cdots, n\right\}$, with none of them being empty. Fix stable rational curves $C_I, C_J, C_K, C_L$, each with $|I|+1,|J|+1,|K|+1,|L|+1$ marked points. Define $F(I,J,K,L)$ as the image of the map $\overline{\mathrm{M}}_{0,4}\to \overline{\mathrm{M}}_{0,n}$ obtained by attaching $C_I, C_J, C_K, C_L$ to the corresponding points. If we choose $C_I, C_J, C_K, C_L$ to be stable pointed rational curves with the maximal number of nodes, then $F(I,J,K,L)$ is an F-curve, and every F-curve can be represented in this form. Note that while $F(I,J,K,L)$ depends on the choice of  $C_I, C_J, C_K, C_L$, its class in the Chow group $\text{A}_1\left(\overline{\mathrm{M}}_{0,n}\right)$ is independent of the choice of specific curves. Therefore, in this paper, we use the notation $F(I,J,K,L)$ without specifying the choices $C_I, C_J, C_K, C_L$. It is important to note that, although the F-conjecture remains open, F-curves generate the Chow group $\text{A}_1(\overline{\mathrm{M}}_{0,n})$.

\subsection{Coinvariant divisors}\label{subsec:voa} 

In this section, we provide a brief introduction to the theory of conformal blocks. While the theory has recently been extended to a more general setting—namely, vertex operator algebras, as developed in \cite{BFM91, FBZ04, NT05, DGT21, DGT22a, DGT22b}—we restrict our attention to the case of Lie algebras, which is the only setting relevant to this paper. The main reference for this section is \cite{TUY89} and our exposition primarily follows \cite{Fak12}.

Let $\mathfrak{g}$ be a simple Lie algebra. The corresponding affine Lie algebra is defined by
\[
\hat{\mathfrak{g}} = \left(\mathfrak{g} \otimes \mathbb{C}(\!(z)\!) \right) \oplus \mathbb{C}c.
\]
Fix $l \in \mathbb{N}$, which we refer to as the \textbf{level}. Define $P_l$ to be the set of dominant integral weights $\lambda$ of $\mathfrak{g}$ such that $\lambda(\theta) \le l$, where $\theta$ is the highest root. For any $\lambda \in P_l$, there exists an irreducible module $W^{\lambda}$ of $\hat{\mathfrak{g}}$, constructed as the irreducible quotient of a Verma module corresponding to $\lambda$, where $c$ acts by multiplication by $l$. Let $S_l$ be the set of such modules.

Given a collection of weights $\lambda_1, \ldots, \lambda_n \in P_l$—equivalently, modules $W^{\lambda_1}, \ldots, W^{\lambda_n}\in S_l$—there exists a corresponding vector bundle $\mathbb{V}_{g,n}(\mathfrak{g},l, W^{\lambda_1}, \ldots, W^{\lambda_n})$ on $\overline{\mathcal{M}}_{g,n}$, known as the \textbf{sheaf of coinvariants}. Its first Chern class, or determinant line bundle, $\mathbb{D}_{g,n}(\mathfrak{g},l,W^{\lambda_1}, \ldots, W^{\lambda_n})$ is called the \textbf{coinvariant divisor}. These were first rigorously defined in \cite{TUY89}. There is a vast body of literature on the study of coinvariant divisors and their relation to the geometry of $\M{n}$. Given the extensive volume, we list only some of the works that are directly related to this paper: \cite{Fak12, AGSS12, GG12, AGS14}.

\begin{rmk}\label{rmk:cdiv}
    In earlier documents, coinvariant divisors were often referred to as conformal block divisors, e.g. \cite{Fak12, BGM15, BGM16}. However, this terminology may be confusing, as they are the first Chern class of the sheaf of coinvariants, not of the sheaf of conformal blocks, which is the dual of the sheaf of coinvariants. Consequently, we adopt this more accurate term.
\end{rmk}

Following \cite{DGT21, DGT22a, DGT22b}, we introduce the following notation. We will write 
\[\mathbb{V}_{g,n}(\mathfrak{g}, l, W^{\lambda_1} \otimes \cdots \otimes W^{\lambda_n})\]
instead of $\mathbb{V}_{g,n}(\mathfrak{g}, l, W^{\lambda_1}, \ldots, W^{\lambda_n})$ to emphasize the role of the tensor product in the construction. Indeed, at each point, the fiber of this vector bundle is defined as a quotient of $W^{\lambda_1} \otimes \cdots \otimes W^{\lambda_n}$. For simplicity, we will use the abbreviation
\[
W^{\bullet} = W^{\lambda_1} \otimes \cdots \otimes W^{\lambda_n}
\]
when the weights $\lambda_1, \ldots, \lambda_n$ are clear from the context. The same notation also applies to coinvariant divisors.

We can define a tensor product between modules in $S_l$, and the resulting $\hat{\mathfrak{g}}$-module decomposes as a direct sum of modules in $S_l$. In particular, this endows $S_l$ with the structure of a fusion ring. This structure allows for the definition of fusion coefficients, dual modules, and related representation-theoretic concepts. Note that $0\in P_l$, and the corresponding module in $S_l$ plays the role of the trivial representation. A key fact we will use frequently in this paper is that if $\mathfrak{g}$ is of type A, D, or E, then $S_1$ forms a group under the tensor product. In the case where $\mathfrak{g} = \mathfrak{sl}_{m}$, the set $S_1$ is isomorphic to $\mathbb{Z}/m\mathbb{Z}$.

We now recall two fundamental theorems concerning the sheaf of coinvariants, which will be used repeatedly throughout this paper. These results are particularly important, as they imply that the sheaf of coinvariants—and consequently the coinvariant divisors—are well-behaved under tautological morphisms.

\begin{thm}\label{thm:propvac}(Propagation of Vacua, \cite[Proposition 2.2.3]{TUY89})
The following identity holds on $\overline{\mathcal{M}}_{g,n+1}$:
\[
\pi_{n+1}^\ast \mathbb{V}_{g,n}(\mathfrak{g}, l, W^\bullet) = \mathbb{V}_{g,n+1}(\mathfrak{g}, l, W^\bullet \otimes W^0),
\]
where $W^0$ denotes the $\hat{\mathfrak{g}}$-module corresponding to the zero weight. In particular,
\[
\pi_{n+1}^\ast \mathbb{D}_{g,n}(\mathfrak{g}, l, W^\bullet) = \mathbb{D}_{g,n+1}(\mathfrak{g}, l, W^\bullet \otimes W^0),
\]
\end{thm}

\begin{thm}\label{thm:factor}(Factorization Theorem, \cite[Theorem 6.2.6]{TUY89})
Let $\xi\colon \overline{\mathcal{M}}_{g_1,n_1} \times \overline{\mathcal{M}}_{g_2,n_2} \to \overline{\mathcal{M}}_{g,n}$ be the clutching map. Then there exists a canonical isomorphism
\[
\xi^\ast \mathbb{V}_{g,n}(\mathfrak{g}, l, W^{\bullet}) \simeq \bigoplus_{M \in S_l} \pi_1^\ast \mathbb{V}_{g_1,n_1}(\mathfrak{g}, l, W^{\bullet} \otimes M) \otimes \pi_2^\ast \mathbb{V}_{g_2,n_2}(\mathfrak{g}, l, W^{\bullet} \otimes M'),
\]
where $M'$ denotes the dual module. In particular,

\begin{align*}
    &\xi^\ast \mathbb{D}_{g,n}(\mathfrak{g}, l, W^{\bullet}) \simeq \\&\bigotimes_{M \in S_l} \pi_1^\ast \mathbb{D}_{g_1,n_1}(\mathfrak{g}, l, W^{\bullet} \otimes M)^{\otimes \rank \mathbb{V}_{g_2,n_2}(\mathfrak{g}, l, W^{\bullet} \otimes M')} \otimes \pi_2^\ast \mathbb{D}_{g_2,n_2}(\mathfrak{g}, l, W^{\bullet} \otimes M')^{\otimes \rank \mathbb{D}_{g_1,n_1}(\mathfrak{g}, l, W^{\bullet} \otimes M)}.
\end{align*}

Similarly, for the clutching map $\xi_{\mathrm{irr}} \colon \overline{\mathcal{M}}_{g-1,n+2} \to \overline{\mathcal{M}}_{g,n}$, we have a canonical isomorphism
\[
\xi_{\mathrm{irr}}^\ast \mathbb{V}_{g,n}(\mathfrak{g}, l, W^{\bullet}) \simeq \bigoplus_{M \in S_l} \mathbb{V}_{g-1,n+2}(\mathfrak{g}, l, W^{\bullet} \otimes M \otimes M').
\]
In particular,
\[
\xi_{\mathrm{irr}}^\ast \mathbb{D}_{g,n}(\mathfrak{g}, l, W^{\bullet}) \simeq \bigotimes_{M \in S_l} \mathbb{D}_{g-1,n+2}(\mathfrak{g}, l, W^{\bullet} \otimes M \otimes M').
\]
\end{thm}

Note that the corresponding statements for  $\mathbb{D}_{g,n}(\mathfrak{g}, l, W^{\bullet})$ in \cref{thm:propvac} and \cref{thm:factor} follow from those for $\mathbb{V}_{g,n}(\mathfrak{g}, l, W^{\bullet})$, together with basic properties of the first Chern class.

Another interesting property is the global generation of the sheaf of coinvariants.

\begin{thm}\label{thm:glogen}(\cite[Lemma 2.5]{Fak12})
    $\mathbb{V}_{g,n}\left(\mathfrak{g}, l, W^{\bullet} \right)$ is globally generated. Hence, $\mathbb{D}_{g,n}\left(\mathfrak{g}, l, W^{\bullet} \right)$ is base point free.
\end{thm}

As stated, the set of simple modules $S_1$ of level 1 over $\mathfrak{sl}_{m}$ forms an abelian group isomorphic to $\mathbb{Z}/m\mathbb{Z}$. Henceforth, we will denote them using nonnegative integers less than $m$. If $W^i$ corresponds to $0 \leq a_i <m$, we denote by $\mathbb{V}_{g,n}^{m}(a_\bullet)$ the vector bundle $\mathbb{V}_{g,n}(\mathfrak{sl}_{m}, 1, W^{\bullet})$ and $\mathbb{D}_{g,n}^{m}(a_\bullet)$ the divisor $\mathbb{D}_{g,n}(\mathfrak{sl}_{m}, 1, W^{\bullet})$, where
\[ (a_\bullet)=(a_1,\cdots, a_n).\]
The $\mathbb{D}_{g,n}^{m}(a_\bullet)$'s are referred to as the \textbf{type A, level 1 coinvariant divisors}. This family of coinvariant divisors has been extensively studied and is comparatively well understood. Now, we will state some important properties of type A coinvariant divisors that will be utilized in this paper, noting that many of them originate from Fakhruddin's fundamental paper \cite{Fak12}.

\begin{thm}\label{thm:lattice}(\cite[Section 5.2]{DGT22b})
    We have
    \[ \text{rank }\mathbb{V}_{g,n}^{m}(a_\bullet)=\begin{cases}
        m^g &\text{if }\sum a_i=0\bmod m\\
        0 &\text{otherwise}.
    \end{cases} \]
    In particular, if $g=0$, then by \cref{thm:factor},
    \begin{equation}\label{eqn:lattice}
        \mathbb{D}_{0,n}^{m}(a_\bullet)\cdot F(I,J,K,L)=\text{deg }\mathbb{D}_{0,4}^{m}\left(\sum_{i\in I}a_i, \sum_{i\in J}a_i,\sum_{i\in K}a_i,\sum_{i\in L}a_i  \right)
    \end{equation}
\end{thm}

We briefly outline the reasoning behind the second statement of \cref{thm:lattice}. The curve $F(I, J, K, L)$ is the image of a morphism $F \colon \M{4} \to \M{n}$ that glues four fixed stable rational curves, each with $|I|+1$, $|J|+1$, $|K|+1$, and $|L|+1$ marked points, respectively. Hence, by \cref{thm:factor}, we have:
\begin{align*}
\mathbb{D}_{0,n}^m(a_{\bullet}) \cdot F(I, J, K, L) 
&= \deg F^\ast \mathbb{D}_{0,n}^m(a_{\bullet}) \\
&= \sum_{0 \le a_I, a_J, a_K, a_L < m} \prod_{X \in \{I, J, K, L\}} \rank \mathbb{V}_{0, |X|+1}^m(a_{\bullet}, -a_X) \cdot \deg \mathbb{D}_{0,4}^m(a_I, a_J, a_K, a_L) \\
&= \deg \mathbb{D}_{0,4}^m\left(\sum_{i \in I} a_i, \sum_{i \in J} a_i, \sum_{i \in K} a_i, \sum_{i \in L} a_i\right).
\end{align*}
The final equality follows from \cref{thm:lattice}. Also note that a more general version of the second statement appears in \cite[Proposition 2.7]{Fak12}, which also follows from a similar argument.

\begin{thm}\label{thm:confcal}(\cite[Lemma 5.1]{Fak12})
    For $0\le a_1\le a_2\le a_3\le a_4<m$,
    \[ \text{deg }\mathbb{D}_{0,4}^m(a_1,a_2,a_3,a_4)=\begin{cases}
        a_1 & \text{if }a_1+a_2+a_3+a_4=2m, a_2+a_3\ge a_1+a_4 \\
        m-a_4 &\text{if }a_1+a_2+a_3+a_4=2m, a_2+a_3\le a_1+a_4\\
        0 &\text{otherwise}.
    \end{cases} \]
\end{thm}

These two properties allow us to compute the intersection of an F-curve and a coinvariant divisor.

\begin{thm}\label{thm:basis}(\cite[Theorem 4.3]{Fak12})
        $\mathbb{D}_{0,n}^2\left(a_\bullet\right)$ is nontrivial if and only if $\sum_{i=1}^{n}a_i$ is even and at least four of the $a_i$'s are $1$. Furthermore, the nontrivial $\mathbb{D}_{0,n}^2\left(a_\bullet\right)$'s form a basis of $\text{Pic}\left(\M{n}\right)_{\Q}$. 
\end{thm}

This theorem provides a basis of the Picard group of $\M{n}$ which behaves well under the projection and clutching maps between moduli space of stable rational curves.

\subsection{Semiample Line Bundles}

The following simple proposition is not used in this paper, but it reveals a connection between the main theorems (\cref{thm:charcont} and \cref{thm:Knumori}) and the semiampleness of certain line bundles on $\M{n}$. It thus provides additional context for our main results.

\begin{prop}\label{prop:nefsem}
Let $X$ be a normal $\mathbb{Q}$-factorial variety, $f:X\to Y$ a contraction and $F=\NE{f}$. If $\mathrm{Pic}(Y) = F^\perp$, then any element in the interior of $F^\perp \cap \mathrm{Nef}(X)$ is semiample. Moreover, if $F^\perp = \mathrm{span}\,\left( F^\perp \cap \mathrm{Nef}(X)\right)$ (for example, if $\mathrm{Nef}(X)$ is polyhedral), then the converse also holds.
\end{prop}

We note that here “interior” refers to the relative interior of $F^\perp \cap \mathrm{Nef}(X)$ as a face of the nef cone, not as a subset of $\mathrm{Pic}(X)$.

\begin{proof}
The condition implies $F^\perp \cap \mathrm{Nef}(X) = f^\ast \mathrm{Nef}(Y)$. Hence, its interior consists of pullbacks of ample divisors on $Y$, which are semiample. For the converse, under the assumption $F^\perp = \mathrm{span}\, F^\perp \cap \mathrm{Nef}(X)$, it suffices to show that any line bundle $L$ lying in the interior of $F^\perp \cap \mathrm{Nef}(X)$ belongs to $f^\ast \mathrm{Pic}(Y)$. Since $L$ is semiample by assumption, it defines a contraction $f': X \to Y'$. Because $L \in F^\perp$, the morphism $f'$ contracts precisely the curves in $F$. Thus, by \cite[Proposition 1.14 (b)]{Deb01}, $f'$ factors through $f$, and in particular, $L \in f^\ast \mathrm{Pic}(Y)$.
\end{proof}

\section{Characterizing Line Bundles on Special Contractions via Contracting Curves}\label{sec:char}

This section is devoted to the characterization of line bundles on certain special contractions of $\M{n}$ via coinvariant divisors. In \cref{subsec:cont}, we introduce three important contractions: Knudsen's, Keel's, and Kapranov's constructions. Moreover, we consider a common generalization of Knudsen's and Keel's constructions, given by the fiber product of two projections. We also provide further details on this construction. In \cref{subsec:proof}, we establish a generalized version of \cref{thm:charcont} for the contractions introduced above. 

\begin{rmk}\label{rmk:charp}
    The main theorems of this section are \cref{thm:charKap}, \cref{thm:charKnu}, and \cref{thm:chargen}. The proofs of these main theorems utilize coinvariant divisors, which are only defined when the base field is $\mathbb{C}$. However, the main theorems also hold over an arbitrary base field. These theorems are purely statements about the Chow ring of $\M{n}$, which does not depend on the base field, as established by \cite{Ke92}. Therefore, if the main theorems hold over $\mathbb{C}$, then they also hold over any field. Accordingly, we will work over $\mathbb{C}$ for the remainder of this chapter.
\end{rmk}

\subsection{Contractions of \texorpdfstring{$\M{n}$}{TEXT}}\label{subsec:cont}
We begin our discussion by recalling the contractions $f_{\text{Knu}}$, $f_{\text{Keel}}$, and $f_{\text{Kap}}$. The definitions of $f_{\text{Knu}}$ and $f_{\text{Keel}}$ are straightforward. Let $\pi_{n-1},\pi_n:\M{n}\to \M{n-1}$ be the projection maps, as described in \cref{subsec:nota}. The product of these maps defines $\pi_{n-1}\times\pi_n:\M{n}\to \M{n-1}\times\M{n-1}$. Note that the image of $\pi_{n-1}\times\pi_n$ comes from a stable rational curve with $n$ marked points, which remembers the first $n-2$ points. Therefore, this morphism factors through $\M{n-1}\times_{\M{n-2}}\M{n-1}$. We denote the induced map as $f_{\text{Knu}}:\M{n}\to\M{n-1}\times_{\M{n-2}}\M{n-1}$. This is \textbf{Knudsen's construction}  as described in \cite{Knu83}. Restricting $f_{\text{Knu}}$ to the locus parametrizing smooth objects of both the domain and codomain shows that it is birational, and indeed, in \cite[Definition 2.3]{Knu83}, Knudsen realized $\M{n}$ as a blow up of $\M{n-1}\times_{\M{n-2}}\M{n-1}$, with $f_{\text{Knu}}$ being the corresponding map. 

\textbf{Keel's construction}, as described in \cite{Ke92}, is a birational morphism defined by the product of projection maps $f_{\text{Keel}}=\pi_n\times \pi_{\left\{1,2,3,n\right\}}:\M{n}\to\M{n-1}\times\M{4}$. The \textbf{Kapranov's construction} $f_{\text{Kap}}:\M{n}\to \mathbb{P}^{n-3}$, described in \cite{Ka93}, is more complex to define. Kapranov fixes $n$ points $p_1,\cdots, p_n$ of $\mathbb{P}^{n-2}$ in general position, and identifies the (closure of) the moduli space of Veronese curves passing through $p_1,\cdots, p_n$, which is a subvariety of the Hilbert scheme of curves in $\mathbb{P}^{n-2}$, with $\M{n}$. The morphism $f_{\text{Kap}}:\M{n}\to \mathbb{P}^{n-3}$ is then defined by sending a Veronese curve to its tangent space of $p_n$.

The codomains of these morphisms are normal varieties. This is straightforward for $f_{\text{Keel}}$ and $f_{\text{Kap}}$ since their codomains are smooth. Although $\M{n-1}\times_{\M{n-2}}\M{n-1}$ is not a smooth variety, its singularities locally resemble a product of an affine space and the cone over a smooth quartic surface, as described in \cite[Section 1]{Ke92}. Hence $\M{n-1}\times_{\M{n-2}}\M{n-1}$ is normal. As mentioned, $f_{\text{Knu}}$, $f_{\text{Keel}}$, and $f_{\text{Kap}}$ are birational; therefore, by the Zariski Main Theorem, they are contractions of $\M{n}$.

Define the set of F-curves
\begin{align*}
    &F_{\text{Kap}}:= \left\{F(I,J,K,L)\ |\ n\in K\text{ and }|K|>1  \right\},\\
    &F_{\text{Keel}}:= \left\{F(I,J,K,L)\ |\ K=\left\{n\right\}, 1,2,3\not\in L \right\},\\
    &F_{\text{Knu}}:= \left\{F(I,J,K,L)\ |\ K=\left\{n\right\}, L=\left\{n-1\right\}  \right\}.
\end{align*}
One can verify that these constitute precisely the F-curves contracted by $f_{\text{Kap}}$, $f_{\text{Keel}}$, and $f_{\text{Knu}}$.

We next introduce a general set of contractions, which contains $f_{\text{Keel}}$ and $f_{\text{Knu}}$ as special cases. Let $S,T$ be the subsets of $\left\{1,2,\cdots, n\right\}$ such that $|S|,|T|\ge 3$. As in the case of Knudsen's and Keel's constructions, the product of projection maps $\pi^S, \pi^T$ defines a morphism $f_{S,T}:\M{n}\to \M{S}\times_{\M{S\cap T}}\M{T}$. If $|S\cap T|\le 2$, then define $\M{S\cap T}=\text{Spec }k$.

The morphisms $f_{S,T}$ are not originally defined in this paper; they have been studied in preceding papers. For instance, analogous morphisms for the higher genus case are considered in \cite[proof of Theorem 0.9]{GKM02}. Then, \cite{GF03} studied them and proved that any fibration of $\M{6}$ factors through $\M{6}\to \M{4}\times \M{4}$. We remark that the set of $F$-curves contracted by $f_{S,T}$ is given by
\[ F_{S,T}:= \left\{F(I,J,K,L)\ |\ I\subseteq S^c\cap T^c \text{ or }I\subseteq S^c, J\subseteq T^c  \right\}. \]

\begin{prop}\label{thm:gencont}
    $f_{S,T}$ is a contraction.
\end{prop}

\begin{proof}
    \textbf{Step 1. }$\M{S} \times_{\M{S \cap T}} \M{T}$ is a local complete intersection, and therefore Cohen-Macaulay.
    
    We have a pullback diagram
    \[\begin{tikzcd}
    \M{S}\times_{\M{S\cap T}}\M{T} \arrow[r, "i"]\arrow[d]& \M{S}\times\M{T}\arrow[d, "\pi^{S\cap T}\times \pi^{S\cap T}"]\\
    \M{S\cap T} \arrow[r, "\Delta"]& \M{S\cap T}\times \M{S\cap T}.
    \end{tikzcd}\]
    The smoothness of $\M{S\cap T}$ implies that $\Delta$ is a regular embedding. Since $\pi^{S\cap T}\times \pi^{S\cap T}$ is flat, $i$ is also a regular embedding. Therefore, the smoothness of $\M{S}\times\M{T}$ ensures that $\M{S} \times_{\M{S \cap T}} \M{T}$ is a local complete intersection.

    \textbf{Step 2. }For any projection $\pi_i: \M{n} \to \M{n-1}$, there exists a closed subscheme $Z \subseteq \M{n}$ of codimension $2$ such that $\pi_i$ is smooth on $\M{n} \setminus Z$. 
    
    For any partition $\left\{1,2,\cdots, n\right\} \setminus \left\{i\right\} = I \amalg J$ such that $|I|, |J| \ge 2$, define $\xi_{I,J}: \M{I+\circ} \times \M{J+\bullet} \times \M{i+\circ+\bullet} \to \M{n}$ to be the clutching map identifying $\circ, \bullet$. Let $Z$ be the union of the images of $\xi_{I,J}$'s for every partition of $I, J$. Then, the fiber of the map $\pi_i: \M{n} \setminus Z \to \M{n-1}$ at $x \in \M{n-1}$ corresponds to the smooth locus of the stable curve corresponding to $x$. Since $\pi_i$ is a flat morphism with smooth fibers, it is a smooth morphism.

    \textbf{Step 3.} $\M{S}\times_{\M{S\cap T}}\M{T}$ is normal.
    
    We use induction on $|T\setminus S|$. If $T\subseteq S$, then there is nothing to prove. Assume the statement for $|T\setminus S|=n$ and prove the case $|T\setminus S|=n+1$. Choose $i\in T\setminus S$. Then we have a pullback square
    \[\begin{tikzcd}
    \M{S}\times_{\M{S\cap T}}\M{T} \arrow[r]\arrow[d, "\pi_i"]& \M{T}\arrow[d, "\pi_i"]\\
    \M{S}\times_{\M{S\cap T}}\M{T-i} \arrow[r]& \M{T-i}.
    \end{tikzcd}\]
    Since $\M{S}\times_{\M{S\cap T}}\M{T-i} \to \M{T-i}$ is flat, by Step 2, there exists a closed subscheme $Z\subseteq \M{S}\times_{\M{S\cap T}}\M{T}$ of codimension 2 such that $\pi_i: \M{S}\times_{\M{S\cap T}}\M{T}\setminus Z\to \M{S}\times_{\M{S\cap T}}\M{T-i}$ is smooth. By the induction hypothesis, $\M{S}\times_{\M{S\cap T}}\M{T-i}$ is normal, so $\M{S}\times_{\M{S\cap T}}\M{T}\setminus Z$ is also normal. Since $Z$ is of codimension $2$, $\M{S}\times_{\M{S\cap T}}\M{T}$ is regular in codimension $1$. By Step 1 and Serre's criterion for normality, $\M{S}\times_{\M{S\cap T}}\M{T}$ is also normal. 

    \textbf{Step 4.} $f_{S,T}:\M{n}\to \M{S}\times_{\M{S\cap T}}\M{T}$ is a contraction.
    
    Since the projections $\pi_i$ are obviously contractions, we may assume that $S\cup T=\left\{1,2,\cdots, n\right\}$. First, assume $|S\cap T|\ge 3$. Then, as we did at the beginning of this subsection, $f_{S,T}$ is birational. Hence, by the Zariski Main Theorem, and Step 3, $f_{S,T}$ is a contraction. Now assume $|S\cap T|\le 2$. Since $|S|\ge 3$, choose $R\subseteq S\setminus T$ so that $|R\cup (S\cap T)|=3$. Then $f_{S,T}$ can be decomposed into $\M{n}\to \M{S}\times_{\M{R\cup (S\cap T)}}\M{R\cup T}\to \M{S}\times_{\M{S\cap T}}\M{T}$. Since $|R\cup (S\cap T)|=3$, the first map is a contraction. Note that, by $|R\cup (S\cap T)|=3$, these fiber products are indeed just products. Hence, the second map is a base change of $\M{R\cup T}\to \M{T}$ by $\M{S}\to \text{Spec }k$. Since the former one is a contraction and the latter one is flat, the second map is a contraction. Hence, $f_{S,T}$ is a composition of contractions, so it is indeed a contraction.    
\end{proof}

\subsection{Proof of  \texorpdfstring{\cref{thm:charcont}}{TEXT}}\label{subsec:proof}

The main goal of this section is to prove \cref{thm:charcont}. We analyze the cases of $f_{\text{Kap}}$ and $f_{S,T}$ separately. As noted in \cite[Section 2.C]{Ka93}, the morphism $f_{\text{Kap}} \colon \M{n} \to \mathbb{P}^{n-3}$ corresponds to the line bundle $\psi_n$. Thus, to prove the theorem for $f_{\text{Kap}}$, it suffices to characterize $\psi_n$ via its intersection with $F_{\text{Kap}}$, which is accomplished in \cref{thm:charKap}. For $f_{S,T}$, we will show in \cref{thm:chargen} that the following sequence is exact:
\[
\text{Pic}(\M{S}) \times \text{Pic}(\M{T}) \xrightarrow{\pi^{S,\ast}+ \pi^{T,\ast}} \text{Pic}(\M{n}) \xrightarrow{c} \Z^{F_{S,T}},
\]
where $c$ denotes the intersection pairing. Since $\M{S} \times_{\M{S \cap T}} \M{T}$ is the image of the morphism $\M{n} \to \M{S} \times \M{T}$, this exactness is sufficient to deduce \cref{thm:charcont}.

First, we characterize line bundles on Kapranov's construction via curves it contracts. To prove \cref{thm:charKap}, we use the basis in \cref{thm:basis}. Let $A_n$ be the set of elements of $\Z_2^n$ whose number of $1$'s is even and at least $4$. Then by \cref{thm:basis}, 
    \[ \left\{\mathbb{D}_{0,n}^2\left(a_\bullet\right)\ |\ a_\bullet\in A_n \right\} \]
is a basis of $\text{Pic}(\M{n})_\Q$.

\begin{thm}\label{thm:charKap}
    Let $\mathcal{L}$ be a line bundle on $\overline{\mathrm{M}}_{0,n}$. Then $\mathcal{L}$ is a constant multiple of $\psi_n$ if and only if the restriction of $\mathcal{L}$ to $\Delta_{0, \left\{i,n\right\}}$ is trivial for every $1\le i\le n-1$, or equivalently, its intersection number with F-curves in $F_{\text{Kap}}$ is zero.
\end{thm}

For the definition of $\Delta_{0,\{i,n\}}$, we refer the reader to \cref{subsec:nota}.

\begin{proof}
    As in \cref{subsec:nota}, let $s_i:\M{n-1}\to \M{n}$ be the $i$th universal section. Then $\mathcal{L}$ is trivial on $\Delta_{0, \left\{i,n\right\}}$ if and only if $s_i^\ast\mathcal{L}=0$ for $1\le i\le n-1$. Note that, as mentioned in \cref{subsec:fconj}, if a line bundle on $\M{n}$ intersects trivially with every F-curve, then such a line bundle is trivial. Since $F_{\text{Kap}}$ are exactly the F-curves contained in $\Delta_{0, \left\{i,n\right\}}$ for $1\le i\le n-1$, it is enough to prove : if $s_i^\ast\mathcal{L}=0$ for $1\le i\le n-1$, then $\mathcal{L}$ is a constant multiple of $\psi_n$.
    
    We first demonstrate this for $\text{Pic}(\M{n})\otimes \Q$. That is, if an element $\mathcal{L}\in \text{Pic}(\M{n})\otimes \Q$ satisfies the condition, then $\mathcal{L}$ is a rational multiple of $\psi_n$. We prove this by calculating $\bigcap_{1\le i\le n-1}\text{ker }s_i^\ast$ in $\text{Pic}(\M{n})\otimes \Q$. Since we already know $\psi_n$ is contained in here (see, for example, \cite[Section 3]{AC09}), it suffices to show that $\bigcap_{1\le i\le n-1}\text{ker }s_i^\ast$ is $1$-dimensional.

    We now show that
    \begin{equation}\label{eqn:kap}
         s_i^\ast \mathbb{D}_{0,n}^2(a_1, \ldots, a_n) = \mathbb{D}_{0,n-1}^2(a_1, \ldots, a_{i-1}, a_i + a_n \bmod 2, a_{i+1}, \ldots, a_{n-1})
    \end{equation}
    on $\M{n-1}$. First, note that $s_i$ coincides with the clutching map $\xi \colon \M{n-1} \times \M{3} \to \M{n}$, where $\M{n-1}$ parametrizes points labeled by $[n-1] \setminus \{i\}$ and $\bullet$, and $\M{3}$ parametrizes points labeled by $i$, $n$, and $\circ$. The map $\xi$ glues the points $\bullet$ and $\circ$. Identifying $\M{3}$ with a point, \cref{thm:factor} gives:
    \begin{align*}
    s_i^\ast \mathbb{D}_{0,n}^2(a_1, \ldots, a_n) 
    &= \xi^\ast \mathbb{D}_{0,n}^2(a_1, \ldots, a_n) \\
    &= \mathbb{D}_{0,n-1}^2(a_1, \ldots, a_{i-1}, a_{i+1}, \ldots, a_{n-1})^{\otimes \rank \mathbb{V}_{0,3}^2(0, a_i, a_n)} \\
    &\quad \otimes \mathbb{D}_{0,n-1}^2(a_1, \ldots, a_{i-1}, 1, a_{i+1}, \ldots, a_{n-1})^{\otimes \rank \mathbb{V}_{0,3}^2(0, a_i, a_n+1)} \\
    &= \mathbb{D}_{0,n-1}^2(a_1, \ldots, a_{i-1}, a_i + a_n \bmod 2, a_{i+1}, \ldots, a_{n-1}),
    \end{align*}
    where the last equality follows from \cref{thm:lattice}.
    
    Fix a $\Q$-line bundle $\mathcal{L}\in \bigcap_{1\le i\le n-1}\text{ker }s_i^\ast$. Let 
    \[ \mathcal{L}=\sum_{(a_\bullet)\in A_n}c(a_\bullet)\mathbb{D}_{0,n}^2\left(a_\bullet\right). \]
    From $s_i^\ast \mathcal{L}=0$ and \cref{eqn:kap}, we have
    \begin{equation}\label{eqn:first}
        c(a_1,\cdots, a_{i-1}, a_i, a_{i+1}, \cdots, a_n)=-c(a_1,\cdots, a_{i-1}, a_i+1 \text{ mod }2, a_{i+1}, \cdots, a_n+1 \text{ mod }2).
    \end{equation}
    if $(a_1,\cdots, a_{i-1}, a_i+a_n \text{ mod }2, a_{i+1},\cdots, a_{n-1})\in A_{n-1}$. Applying this relationship twice, we get
    \begin{equation}\label{eqn:second}
        c(a_1,\cdots,a_i,\cdots, a_j\cdots, a_n)=c(a_1,\cdots,a_i+1 \text{ mod }2,\cdots, a_j+1 \text{ mod }2,\cdots, a_n)
    \end{equation}
    if $(a_1,\cdots, a_{i-1}, a_i+a_n \text{ mod }2, a_{i+1},\cdots, a_{n-1}), (a_1,\cdots, a_{j-1}, a_j+a_n \text{ mod }2, a_{j+1},\cdots, a_{n-1})\in A_{n-1}$. Applying these identities multiple times allows us to prove that the dimension of the subspace of line bundles which satisfies \cref{eqn:first} and \cref{eqn:second} is $1$-dimensional, and generated by
    \[ \sum_{(a_1,\cdots, a_n)\in A_n}(-1)^{a_n}\mathbb{D}_{0,n}^2(a_1,\cdots,a_n).\]
    
    \textbf{Case 1. }$\mathbf{n}$\textbf{ is even. }In this case, by applying \cref{eqn:first} and \cref{eqn:second} iteratively, starting from $c(1,1,\cdots,1)$, we establish that
    \[ c(a_1,\cdots, a_n)=(-1)^{a_n+1}c(1,1,\cdots, 1). \]
    for any $(a_1,\cdots, a_n)\in A_n$.  Indeed, this proof can be executed by performing induction on the number of $0$'s in the tuple $(a_1,\cdots, a_n)\in A_n$.

    \textbf{Case 2. }$\mathbf{n}$\textbf{ is odd. }Let $A_n^1$ be the subset of $A_n$ whose number of $0$'s is exactly $1$. Using \cref{eqn:first} on $c(a_1, \cdots, a_n)$ for $(a_1,\cdots, a_n)\in A_n^1$, it is straightforward to show that
    \[ c(a_1,\cdots, a_n)=-c(1,1,\cdots,1, 0). \]
    if $a_n\ne 0$. Now, by \cref{eqn:first} and \cref{eqn:second}, we can prove that
    \[ c(a_1,\cdots, a_n)=(-1)^{a_n+1}c(1,1,\cdots, 1). \]
    for any $(a_1,\cdots, a_n)\in A_n$, again, by induction on number of zeros. 
    
    This completes the proof for $\text{Pic}(\M{n})\otimes \Q$. We now establish the claim for $\text{Pic}(\M{n})$. If $\mathcal{L}$ is a line bundle on $\text{Pic}(\M{n})$ satisfying the condition, then there exists $c\in \Q$ such that $\mathcal{L}=c\psi_n$ in $\text{Pic}(\M{n})\otimes \Q$. Let $F$ be any F-curve not contained in $F_{\text{Kap}}$. By \cite[Section 3]{AC09}, $F\cdot \psi_n=1$. Then $c=c(F\cdot \psi_n)=F\cdot \mathcal{L}$, which is an integer, thus indicating that $\mathcal{L}$ is an integral multiple of $\psi_n$. This conclusion also holds in $\text{Pic}(\M{n})$ since $\text{Pic}(\M{n})$ is a free $\Z$-module.
\end{proof}

Note that \cref{thm:charKap} is a special case of a theorem by V. Alexeev in \cite[Proposition 4.6]{Fak12}. We refer to \cref{subsec:nec} for the relevant discussion. However, the proof of \cref{thm:charKap} is of a very different nature from Alexeev's theorem and deserves its own interest. In particular, as we will see, we can generalize this argument to other cases where Alexeev's argument cannot apply.

\cref{thm:basis} gives a basis of $\text{Pic}(\M{n})\otimes \Q$, though other bases are also known. For example, \cite[Lemma 2]{GF03} provides an integral basis of $\text{Pic}(\M{n})$. However, this basis is not stable under pullback along projections and clutching maps. Hence, if we interpret \cref{thm:charKap} in terms of this basis, the linear algebra problem becomes very intricate. It is difficult to prove \cref{thm:charKap} using this basis, even for small values like $n=7$. The basis provided by coinvariant divisors is useful because it is stable under pullbacks, making the corresponding linear algebra problem relatively easy.

\begin{rmk}\label{rmk:inv2}
    As a byproduct of the proof, we obtain
    \begin{equation}\label{eqn:psi}
        \psi_n = \frac{1}{2^{n-4}} \sum_{(a_1, \ldots, a_n) \in A_n} (-1)^{a_n + 1} \mathbb{D}_{0,n}^2(a_1, \ldots, a_n).
    \end{equation}
    By the proof of \cref{thm:charKap}, it suffices to compare a nontrivial intersection of both sides with a curve. Let $F$ be an $F$-curve not contained in $F_{\text{Kap}}$, i.e., $F = F(I, J, K, \{n\})$ as described above. Then $F \cdot \psi_n = 1$. Note that by \cref{thm:confcal}, $\mathbb{D}_{0,4}^2(a_1, \ldots, a_4)$ equals $1$ if all $a_i$ are $1$, and $0$ otherwise. By \cref{eqn:lattice}, we compute:
    \begin{align*}
    &\sum_{(a_1, \ldots, a_n) \in A_n} (-1)^{a_n + 1} \mathbb{D}_{0,n}^2(a_1, \ldots, a_n) \cdot F
    \\&= \left| \left\{ (a_1, \ldots, a_n) \in A_n \;\middle|\; \sum_{i \in I} a_i = \sum_{i \in J} a_i = \sum_{i \in K} a_i = a_n = 1 \right\} \right|= 2^{n-4}.
    \end{align*}
    Hence, the intersection numbers with $F$ agree on both sides.

    This clearly reveals the following well-known facts:
    \begin{enumerate}
        \item  The cone generated by $\left\{\mathbb{D}_{0,n}^2\left(a_\bullet\right)\ |\ a_\bullet \in A_n \right\}$ is strictly smaller than the nef cone.
        \item $\left\{\mathbb{D}_{0,n}^2\left(a_\bullet\right)\ |\ a_\bullet \in A_n \right\}$ is not a $\Z$-basis of $\text{Pic}(\M{n})$.
    \end{enumerate}
    Note that (1) is analyzed in more detail in \cite{GG12}. Despite (2), \cref{eqn:psi} proves that $\psi_n$ is contained in the submodule generated by $\mathbb{D}_{0,n}^2\left(a_\bullet\right)$'s if we invert $2$. This is not a coincidence; see \cref{thm:CDint}.
\end{rmk}

The proof of an analogue of \cref{thm:charKap} for $f_{S,T}$ also follows a similar two-step process. First, the use of the basis in \cref{thm:basis} converts the theorem on $\text{Pic}(\M{n}) \otimes \mathbb{Q}$ into a relatively easy linear algebra problem. Once this is resolved, the second step involves establishing the result over $\text{Pic}(\M{n})$. \cref{lem:linalg} solves the first step for the case of $f_{S,T}$.

\begin{lem}\label{lem:linalg}
    Let $k$ be a field of characteristic not equal to $2$, let $G = \mathbb{Z}_2^n$, and let $I := \text{Hom}_{\text{Ab}}(G, \mathbb{F}_2)$. For any $\phi \in I$, there exists an induced homomorphism of group rings $\phi_* : k[G] \to k[\mathbb{F}_2]$. Define a function $F : k[G] \to \prod_{\phi \in I} k[\mathbb{F}_2]$ as the product of $\phi_*$ for every $\phi \in I$. Thus, the preimage $F^{-1}\left(\prod_{\phi \in I} k\right)$ consists only of constant functions, i.e., $k$.
\end{lem}

\begin{proof}
    By flatness, we may assume that $k$ is algebraically closed. Since both $k[G]$ and $k[\mathbb{F}_2]$ are Hopf algebras over $k$, they are (non-canonically) isomorphic to coordinate ring of finite group schemes $\mu_2^n$ and $\mu_2$ over $k$, respectively. Additionally, since $\phi_\ast$ is a homomorphism of Hopf algebras, and not just a ring homomorphism, it induces a homomorphism of group schemes $\phi^\ast:\mu_2\to \mu_2^n$. Note that any homomorphism of group scheme from $\mu_2$ to $\mu_2^n$ can be represented as $\phi^\ast$ for some $\phi\in I$. Therefore, the map $F$ corresponds to the morphism of group schemes $\amalg_I \mu_2\to \mu_{2^n}$. The preimage $F^{-1}\left(\prod_{\phi\in I}k\right)$ corresponds to the set of functions on $\mu_2^n$ whose restriction via $\phi^\ast$ is a constant function on $\mu_2$. Since $\phi^\ast$ covers every morphism from $\mu_2$ to $\mu_2^n$, the only functions that satisfy this condition are the constant functions. Hence $F^{-1}\left(\prod_{\phi\in I}k\right)=k$.
\end{proof}

Note that \cref{lem:linalg} is false if $\text{char } k = 2$. This highlights the significance of $\mu_2$ in the proof. Indeed, the reason we need to invert $2$ to express $\psi$-classes as a linear combination of $\mathfrak{sl}_2$ level $1$ coinvariant divisors, as seen in \cref{rmk:inv2}, is that \cref{lem:linalg} does not hold for characteristic $2$. See \cref{thm:CDint} for details.

Now we are ready to prove the main theorem of this section for $\text{Pic}\left(\M{n}\right)_\Q$. The following result is a direct application of \cref{lem:linalg}, which serves as a key technical tool in the proof.

\begin{lem}\label{lem:Qknu}
     For distinct elements $s,t$ in $\left\{1,2,\cdots,n\right\}$, define
    \[ F_{s,t}:= \left\{F(I,J,K,L)\ |\ I=\left\{s\right\}, J=\left\{t\right\}  \right\}. \]
    If a $\Q$-line bundle 
    \[\mathcal{L}=\sum_{(a_\bullet)\in A_n}c(a_\bullet)\mathbb{D}_{0,n}^2\left(a_\bullet\right)\]
    of $\M{n}$ intersects with every element of $F_{s,t}$ trivially, then $c(a_\bullet)=0$ if $a_s=a_t=1$.
\end{lem}

\begin{proof}
    First, we will prove that if $a_s = 0$ or $a_t = 0$, then the intersection of $\mathbb{D}_{0,n}^2(a_\bullet)$ with a curve $F(I, J, \{s\}, \{t\})$ in $F_{s,t}$ is trivial: By \cref{eqn:lattice}, we have
    \[
    \mathbb{D}_{0,n}^2(a_\bullet) \cdot F = \deg \mathbb{D}_{0,4}^2\left( \sum_{i \in I} a_i, \sum_{i \in J} a_i, a_s, a_t \right),
    \]
    which is zero by \cref{thm:confcal}.
    
    Hence, we may assume that
    \[ \mathcal{L}=\sum_{(a_\bullet)\in A_n^{s,t}}c(a_\bullet)\mathbb{D}_{0,n}^2\left(a_\bullet\right) \]
    where $A_n^{s,t}=\left\{(a_\bullet)\in A_n\ |\ a_s=a_t=1 \right\}$. We now translate the situation to one where \cref{lem:linalg} can be applied.
    
    Choose $r\in \left\{1,2,\cdots, n\right\}\setminus  \left\{s,t\right\}$, and let $I:= \left\{1,2,\cdots, n\right\}\setminus  \left\{r,s,t\right\}$. Then there exists a natural one to one correspondence between non-identity element in the group $G=\mathbb{F}_2^I$ and $A_n^{s,t}$: Let $h_1:A_n^{s,t}\to G\setminus \left\{0\right\}$ be the map that sends $\left(a_\bullet\right)$ to $\left(a_\bullet\right)$ mod $2$ and disregards the number at $r$. Since the sum of the entries of an element of $A_n^{s,t}$ is even, $h_1$ is one-to-one. Consequently, a line bundle $\mathcal{L}$ of the preceding form corresponds to an element of the group ring $\mathbb{Q}[G]$, with vanishing constant coefficient. We denote this second association by $(h_1)_\ast$. Moreover, a nonzero group homomorphism $\phi:G\to \mathbb{F}_2$ corresponds to $F_{s,t}$: For any $F(\left\{s\right\},\left\{t\right\},K,L)\in F_{s,t}$, without loss of generality, assume that $r\in K$.  Define $h_2:F_{s,t}\to \text{Hom}(G, \mathbb{F}_2)\setminus \left\{0\right\}$ as $h_2(F(\left\{s\right\},\left\{t\right\},K,L)$ being the map sending $K\setminus \left\{r\right\}$ to $0$ and $L$ to $1$. Since $I=(K\cup L)\setminus \left\{r\right\}$ and $L$ is nonempty, this defines a well-defined nonzero morphism in $\text{Hom}(G, \mathbb{F}_2)$. By mapping $\phi\in \text{Hom}(G, \mathbb{F}_2)\setminus \left\{0\right\}$ to $F(\left\{s\right\},  \left\{t\right\}, (\phi^{-1}(0)\cap I)\cup \left\{r\right\}, (\phi^{-1}(1)\cap I))$, we can demonstrate that $h_2$ is a bijection.

    This defines a commutative diagram
    \begin{equation}\label{eqn:lemma}
        \begin{tikzcd}
        \mathbb{Q}^{A_n^{s,t}}\times F_{s,t}\arrow[r, "c_1"]\arrow[d, "(h_1)_\ast\times h_2"]& \mathbb{Q}\arrow[d, "i"] \\
        \mathbb{Q}[G]\times \text{Hom}(G, \mathbb{F}_2)\arrow[r, "c_2"]&\mathbb{Q}[\mathbb{F}_2]
    \end{tikzcd} 
    \end{equation}
    where $c_1$ is the intersection pairing, $c_2$ is the evaluation map defined by
    \[ \left(\sum_{g\in G}a_g \cdot g, \phi\right)\mapsto \sum_{g\in G} a_g\cdot \phi(g)  \]
    and $i$ be the map defined by $i(q)=q\cdot 1$. By the definition of $h_1$ and $h_2$, and using \cref{thm:lattice} and \cref{thm:confcal}, it is straightforward to prove that this diagram commutes. The statement of the lemma is equivalent to the following: if $c_1(L, -)\in \text{Hom}(F_{s,t},\Q)$ is zero for some $L\in\mathbb{Q}^{A_n^{s,t}}$, then $L=0$. This follows from \cref{eqn:lemma} and \cref{lem:linalg}, since the image of $h_1$ does not contain the identity.     
\end{proof}

As a first step, we consider the projection maps.

\begin{prop}\label{thm:charprojQ}
    For any $1\le i\le n$, define
    \[ F_{i}:= \left\{F(I,J,K,L)\ |\ I=\left\{i\right\}  \right\}, \]
    There exists an exact sequence
    \[ 0\to \text{Pic}(\M{n-1})_\Q\xrightarrow{\pi_{i}^\ast} \text{Pic}(\M{n})_\Q\xrightarrow{c} \Q^{F_{\text{i}}}.\]
    where $c$ is the intersection pairing. In particular,  a $\Q$-line bundle $\mathcal{L}$ on $\M{n}$ is in the image of $\pi_{i}^\ast$ if and only if it intersects every curve in $F_{i}$ trivially.
\end{prop}

\begin{proof}
     $\pi_i^\ast$ is injective since $\pi_i$ is a contraction. Let  
    \[ \mathcal{L}=\sum_{(a_\bullet)\in A_n}c(a_\bullet)\mathbb{D}_{0,n}^2\left(a_\bullet\right). \]
    be a $\Q$-line bundle on $\M{n}$ whose intersection number with any curve in $F_{i}$ is zero. By the Propagation of Vacua, \cref{thm:propvac}, it suffices show that $c(a_\bullet)=0$ if $a_i=1$. Assume, for contradiction, that there exists a counterexample $(a_\bullet)\in A_n$. Given that $(a_\bullet)\in A_n$, there must exist some $j\ne i$ such that $a_j=1$. Note that $F_{i,j}$, defined in \cref{lem:Qknu}, is contained in $F_i$. Therefore, $\mathcal{L}$ intersects with curves in $F_{i,j}$ trivially. This contradicts \cref{lem:Qknu}.
\end{proof}

\cref{thm:chargenQ} is the main theorem for $\text{Pic}\left(\M{n}\right)_\Q$.

\begin{prop}\label{thm:chargenQ}
    Let $S,T\subseteq \left\{1,2,\cdots,n \right\}$ be subsets such that $|S|, |T|\ge 3$ and let $f_{S,T}$ and $F_{S,T}$ as described in \cref{subsec:cont}. 
    
    \begin{enumerate}
        \item The following sequence is exact:
        \[ 0\to \text{Pic}(\M{S\cap T})_\Q \xrightarrow{(\pi^{S\cap T,\ast}, -\pi^{S\cap T,\ast})}\text{Pic}(\M{S})_\Q\times \text{Pic}(\M{T})_\Q\xrightarrow{\pi^{S,\ast}+\pi^{T,\ast}} \text{Pic}(\M{n})_\Q\xrightarrow{c} \Q^{F_{S,T}}\]
        where $c$ is the intersection pairing. In particular, a $\Q$-line bundle $\mathcal{L}$ on $\M{n}$ is in the image of $\pi^{S,\ast}+\pi^{T,\ast}$ if and only if it intersects every curve in $F_{S,T}$ trivially.
        \item The image of  $f_{S,T}^\ast: \text{Pic}\left(\M{S}\times_{\M{S\cap T}}\M{T} \right)_\Q\to \text{Pic}\left(\M{n} \right)_\Q$ is $\text{Ker }c$.
        \item The  natural map $\text{Pic}(\M{S})_\Q \times \text{Pic}(\M{T})_\Q \to \text{Pic}\left(\M{S}\times_{\M{S\cap T}}\M{T}\right)_\Q$ is a surjection.
    \end{enumerate}
\end{prop}

\begin{proof}
    (1) Since $\pi^{S\cap T}$ is a composition of contractions $\pi_i$, it is a contraction. Therefore, $\pi^{S \cap T, \ast}$—and therefore $(\pi^{S \cap T, \ast}, -\pi^{S \cap T, \ast})$—is injective. This shows the exactness at $\text{Pic}(\M{S\cap T})_\Q$. By the same reasoning, both $\pi^{S,\ast}$ and $\pi^{T,\ast}$ are injective. Moreover, by \cref{thm:propvac} and \cref{thm:basis}, their images are $\Q$-vector spaces with bases consisting of coinvariant divisors $\mathbb{D}_{0,n}^2(a_\bullet)$ such that $a_i = 0$ for all $i \in S^c$ (respectively, $i \in T^c$). Thus, applying \cref{thm:propvac} and \cref{thm:basis} again, we conclude that the intersection of their images is precisely the image of $\pi^{S \cap T, \ast}$. Therefore, the sequence is exact at $\text{Pic}(\M{S})_\Q \times \text{Pic}(\M{T})_\Q$.

    We now prove exactness at $\text{Pic}(\M{n})_\Q$. Since every element of $F_{S,T}$ is contracted by both $\pi^S$ and $\pi^T$, the image of $\pi^{S,\ast} + \pi^{T,\ast}$ is contained in the kernel of $c$. By \cref{thm:charprojQ}, we may assume that $\{1,2,\dots,n\} = S \cup T$. Let $\mathcal{L}$ be a $\Q$-line bundle on $\M{n}$ whose intersection number with any curve in $F_{S,T}$ is zero. By \cref{thm:basis},
    \[
    \mathcal{L} = \sum_{(a_\bullet) \in A_n} c(a_\bullet) \mathbb{D}_{0,n}^2(a_\bullet).
    \]
    We claim that $c(a_\bullet) \ne 0$ only if $a_i = 0$ for all $i \in S^c$ or for all $i \in T^c$. Then, \cref{thm:propvac} shows that $\mathcal{L}$ lies in the image of $\pi^{S,\ast}+\pi^{T,\ast}$. Suppose, for contradiction, that there exists $(a_\bullet) \in A_n$ such that $c(a_\bullet) \ne 0$, but $a_i \ne 0$ for some $i \in S^c$ and some $i \in T^c$. Then there exist $s \in S^c$ and $t \in T^c$ such that $a_s = a_t = 1$. Since $S \cup T = \{1,2,\dots,n\}$, we have $s \ne t$. As $F_{s,t} \subseteq F_{S,T}$, the bundle $\mathcal{L}$ intersects trivially with all curves in $F_{s,t}$. This contradicts \cref{lem:Qknu}.

    (2) Since $\pi^S\times \pi^T$ factors through $f_{S,T}$, the image of $f_{S,T}^\ast$ contains the image of $\pi^{S,\ast}+\pi^{T,\ast}$. Let $\mathcal{L}$ be a $\Q$-line bundle on $\M{S}\times_{\M{S\cap T}}\M{T}$. Then, since $f_{S,T}$ contracts curves in $F_{S,T}$, $f_{S,T}^\ast\mathcal{L}$ intersects with curves in $F_{S,T}$ trivially. Hence, $f_{S,T}^\ast\mathcal{L}$ is contained in the image of $\pi^{S,\ast}+\pi^{T,\ast}$ by the first assertion. This demonstrates (2). 
    
    (3) (2) implies that for any $\mathcal{L}\in \text{Pic}\left(\M{S}\times_{\M{S\cap T}}\M{T} \right)_\Q$, there exist $\mathcal{L}_1\in \text{Pic}(\M{S})_\Q, \mathcal{L}_2\in \text{Pic}(\M{T})_\Q$ such that $f_{S,T}^\ast\mathcal{L}=\pi^{S,\ast}\mathcal{L}_1+\pi^{T,\ast}\mathcal{L}_2=f_{S,T}^\ast i^\ast(\mathcal{L}_1\boxtimes \mathcal{L}_2)$ where $i$ is the inclusion 
    \[ i:\M{S}\times_{\M{S\cap T}}\M{T}\to \M{S}\times\M{T}. \]
    Since $f_{S,T}$ is a contraction, $f_{S,T}^\ast$ is an injection. Therefore, $\mathcal{L}=\mathcal{L}_1\boxtimes \mathcal{L}_2$. This proves (3).
\end{proof}

For the reader’s convenience, we state \cref{thm:chargenQ} in the special case of $f_{\text{Knu}}$, where $S = [n-1]$ and $T = [n] \setminus \{n-1\}$. This case is particularly important, as it plays a key role in the proofs of \cref{thm:charKnu} and the results in \cref{sec:cons}. See also \cref{rmk:Knu}.

\begin{cor}\label{thm:charKnuQ}
    \begin{enumerate}
        \item The following sequence is exact
        \[ 0\to \text{Pic}(\M{n-2})_\Q \xrightarrow{(\pi_{n}^\ast, -\pi_{n-1}^\ast)}\text{Pic}(\M{n-1})_\Q\times \text{Pic}(\M{n-1})_\Q\xrightarrow{\pi_{n-1}^\ast+\pi_{n}^\ast} \text{Pic}(\M{n})_\Q\xrightarrow{c} \Q^{F_{\text{Knu}}}\]
        where $c$ is the intersection pairing. In particular, a $\Q$-line bundle $\mathcal{L}$ on $\M{n}$ is in the image of $\pi_{n}^\ast+\pi_{n-1}^\ast$ if and only if it intersects with curves in $F_{\text{Knu}}$ trivially.
        \item The image of  $f_{\text{Knu}}^\ast: \text{Pic}\left(\M{n-1}\times_{\M{n-2}}\M{n-1} \right)_\Q\to \text{Pic}\left(\M{n} \right)_\Q$ is $\text{Ker }c$.
        \item The  natural map $\text{Pic}(\M{n-1})_\Q \times \text{Pic}(\M{n-1})_\Q \to \text{Pic}\left(\M{n-1}\times_{\M{n-2}}\M{n-1} \right)_\Q$ is a surjection.
    \end{enumerate}
\end{cor}

\begin{rmk}
    In \cref{thm:charKnuQ} (1), $c$ is indeed also surjective, according to  \cref{cor:knusln}.
\end{rmk}

To prove these theorems for $\text{Pic}\left(\M{n}\right)$, we use \cref{lem:QtoZ}.

\begin{lem}\label{lem:QtoZ}
    Let $A,B,C$ be $\Z$-modules with morphisms $A\xrightarrow{f} B\xrightarrow{g} C$. If the sequence $A_\Q\xrightarrow{f} B_\Q\xrightarrow{g} C_\Q$ is exact and the cokernel of $f:A\to B$ is torsion-free, then $A\to B\to C$ is exact.
\end{lem}

\begin{proof}
    Let $b\in \ker g$. By the exactness after tensoring $\Q$, there exists $n\in \N$ such that $nb=f(a)$ for some $a\in A$. Since the cokernel of $f$ is torsion-free, $b\in \text{Im }f$.
\end{proof}

\cref{thm:dual} allows for translation between $\text{Pic}$ and $A_1$.

\begin{prop}\label{thm:dual}
    The intersection pairing $A^1\left(\M{n}, \Z \right)\times A_1(\M{n}, \Z)\to \Z$ is a perfect pairing. Moreover, the pairing $\text{Pic}\left(\M{n} \right)\times A_1(\M{n}, \Z)\to \Z$ is also a perfect pairing.
\end{prop}

\begin{proof}
    Note that, according to \cite{Ke92}, 
    \begin{enumerate}
        \item Rational equivalence and numerical equivalence coincide on $\M{n}$.
        \item The cycle class map $A^\ast(\M{n},\Z)\to \text{H}^\ast(\M{n}, \Z)$ is an isomorphism.
        \item $\text{H}^d(\M{n},\Z)$ is a free $\Z$-module of finite rank for any $d\in \N$.
    \end{enumerate}
    In particular, by (1), $\text{Pic}(\M{n})\simeq A^1(\M{n},\Z)$. Hence, it suffices to prove the first statement. According to (2), the paring $A^1\left(\M{n}, \Z \right)\times A_1(\M{n}, \Z)\to \Z$ corresponds to the cup product pairing $\text{H}^{2n-8}(\M{n}, \Z)\times \text{H}^2(\M{n}, \Z)\to \Z$. Since the second pairing is perfect by Poincaré duality, together with the freeness condition in (3), the first one is also perfect.
\end{proof}

When working with $\Q$-coefficients, \cref{thm:dual} is trivial, even for general projective smooth varieties, from the definition of numerical equivalence. However, it does not hold in general for projective smooth varieties with $\Z$-coefficients. Indeed, \cref{thm:dual} holds for a projective smooth variety $X$ over $\C$ if and only if the integral Hodge conjecture for $1$-cycles on $X$ is verified, up to torsion.

\begin{prop}\label{thm:charproj}
    The following sequences
    \[ 0\to \text{Pic}(\M{n-1})\xrightarrow{\pi_i^\ast} \text{Pic}(\M{n})\xrightarrow{c} \Z^{F_i} \]
    and
    \[ \Z^{F_i}\xrightarrow{c^\ast} A_1(\M{n}, \Z)\xrightarrow{\pi_{i,\ast}}  A_1(\M{n-1}, \Z)\to 0 \]
    are exact, where $F_i$ is the set of F-curves defined in \cref{thm:charprojQ}, $c^\ast$ realizes $F_{i}$ as $1$-cycles, and $c$ is the intersection pairing with $F_{i}$.
\end{prop}

\begin{proof}
    First, we prove the exactness of the first sequence. The map $\pi_i^\ast$ is injective since $\pi_i$ is a contraction. For the rest of the sequence, by \cref{thm:charprojQ} and \cref{lem:QtoZ}, it is sufficient to show that the cokernel of $\pi_i^\ast$ is torsion-free. This condition is met since $\pi_i:\M{n}\to\M{n-1}$ has a section $s:\M{n-1}\to\M{n}$ and the cokernel of $\pi_i^\ast$ coincides with the kernel of $s^\ast$. Hence, the first sequence is exact.

    We now aim to prove the exactness of the second sequence by dualizing the first. However, there is an obstruction: if the cokernel of $c$ is not torsion-free, then the dual of the first exact sequence may fail to be exact. Notably, this condition is also sufficient: if the cokernel of $c$ is torsion-free, then the first exact sequence can be decomposed into two short exact sequences
    \[
    0 \to \text{Pic}(\M{n-1}) \xrightarrow{\pi_i^\ast} \text{Pic}(\M{n}) \xrightarrow{c} \text{Im }c \to 0
    \quad \text{and} \quad
    0 \to \text{Im }c \to \Z^{F_i} \to \text{Coker }c \to 0.
    \]
    Since all modules involved are free $\Z$-modules, these sequences remain exact after taking duals. By combining their duals, we obtain the second exact sequence via \cref{thm:dual}. Therefore, it suffices to prove that $\text{coker }c$ is torsion-free.

    Note that the rank of $\text{Pic}(\M{n})$ is $2^{n-1}-\binom{n}{2}-1$ and the rank of $\text{Pic}(\M{n-1})$ is $2^{n-2}-\binom{n-1}{2}-1$, so the rank of the cokernel of $c$ is $2^{n-2}-(n-2)-1$. Hence, it suffices to show that the image of $c$ contains a sub-$\Z$-module of rank $N:=2^{n-2}-(n-2)-1$, whose inclusion into $\Z^{F_i}$ has a torsion-free cokernel. We prove this by demonstrating the following: there exists line bundles $\delta_1,\cdots, \delta_N\in \text{Pic}(\M{n})$ and $F$-curves $F_1,\cdots, F_N\in F_i$ such that the determinant of the intersection matrix $(\delta_k\cdot F_l)_{1\le k,l\le N}$ is $1$. If this is true, then the image of $\delta_1,\cdots, \delta_N$ under $c$ generates the desired submodule.

    Fix $j\in \left\{1,2,\cdots, n\right\}\setminus\left\{i\right\}$. Let $J=\left\{1,2,\cdots, n\right\}\setminus\left\{i,j\right\}$ and let $I$ be the set of subsets of $J$ that consist of at least $2$ elements. Then $|I|=N$, so we use $I$ as an indexing set. For $A\in I$, define $\delta_A$ be the boundary divisor of $\M{n}$ corresponds to the clutching map $\xi:\M{A+\bullet}\times\M{A^c+\bullet}\to \M{n}$. Since $i,j\in A$ and $|A|\ge 2$, this is well-defined. Also, let $F_A=F(\left\{i\right\}, \left\{\text{min }A\right\},A\setminus\left\{\text{min }A\right\}, (A\cup i)^c )$. We remark that the specific choice of $\text{min }A$ plays no role here; we just want to choose an element from each $A$. Choose a bijection $h:\left\{1,2,\cdots, N\right\}\to I$ so that if $k<l$ then $|h(k)|\ge |h(l)|$. For $1\le k,l\le N$, let $\delta_k:=\delta_{h(k)}$ and $F_l:=F_{h(l)}$. According to \cite[Lemma 4.3]{KM13},
    \[ \delta_k\cdot F_k=1,\ \delta_k\cdot F_l=0\text{ if }k<l. \]
    Hence, the matrix $(\delta_k\cdot F_l)_{1\le k,l\le N}$ is an upper triangular matrix whose diagonal entries are 1, so the determinant is $1$. An example for the case $n=5$, $i=5$, and $j=4$ is given below:
    \begin{center}
    \begin{tabular}{ c | c | c | c | c }
     & $\delta_1=\delta_{\left\{1,2,3\right\}}$ & $\delta_2=\delta_{\left\{1,2\right\}}$ & $\delta_3=\delta_{\left\{1,3\right\}}$ & $\delta_4=\delta_{\left\{2,3\right\}}$ \\ \hline
    $F_1=F(5,1, \left\{2,3\right\}, 4)$ & $1$ & $0$ & $0$ & $-1$ \\  \hline
    $F_2=F(5,1,2,\left\{3,4\right\})$   & $0$ & $1$ & $0$ & $0$ \\   \hline
    $F_3=F(5,1,3,\left\{2,4\right\})$   & $0$ & $0$ & $1$ & $0$ \\ \hline
    $F_4=F(5,2,3,\left\{1,4\right\})$   & $0$ & $0$ & $0$ & $1$
    \end{tabular}.
    \end{center}
    This demonstrates that the cokernel of $c$ is torsion-free.
\end{proof}

For \cref{thm:charKnu}, we need the following technical lemma.

\begin{lem}\label{lem:partition}
Let $X$ be a finite set and $A$ a nonempty proper subset of $X$. There exists a function $a: X \to \mathbb{N}$ such that for any subset $B \subseteq X$, $\sum_{b \in B} a_b = \sum_{b \not\in B} a_b$ holds if and only if $B = A$ or $B = A^c$.
\end{lem}

\begin{proof}
    For any $B\subseteq X$, let
    \[ M=\left\{ a:X\to \R_{>0} \ |\ \sum_{b \in A} a_b = \sum_{b \not\in A} a_b \right\},\ M_B:=\left\{g\in M\ |\ \sum_{b \in B} a_b = \sum_{b \not\in B} a_b \right\}. \]
    Then, $M$ is a $|X|-1$-dimensonal manifold, and $M_B$ is a $|X|-2$-dimensional embedded submanifold for every $B\ne A,A^c$. Since $\Q^{X}\cap M$ is dense in $M$, if we let $N:=M\setminus \cup_{B\ne A,A^c}M_B$, then $N\cap \Q^{X}\ne \emptyset$. Choose $a\in N\cap \Q^{X}$. After multiplying this by the least common multiple of denominators, we may assume that the value of $a$ is contained in $\N$. By construction, $a$ is the desired function. 
\end{proof}

As a consequence, we obtain \cref{cor:knusln}, which plays an important role in the proof of \cref{thm:charKnu} by implying that the intersection pairing
\[
\text{Pic}(\M{n}) \xrightarrow{c} \Z^{F_{\text{Knu}}}
\]
is surjective. Although only a simplified version of \cref{cor:knusln} is necessary for our purposes—for instance, the base point freeness is not necessary, and thus the latter part of the proof could be omitted—we will make use of the full strength of \cref{cor:knusln} in \cref{subsec:Mori}. Therefore, we present the complete proof of the corollary here.

\begin{cor}\label{cor:knusln}
For any $C \in F_{\text{Knu}}$, there exists a base point free line bundle $\mathcal{L}$ on $\M{n}$ such that $\mathcal{L} \cdot C = 1$ and $\mathcal{L} \cdot D = 0$ for every $D \in F_{\text{Knu}}$ not equal to $C$.
\end{cor}

\begin{proof}
    We begin by assuming the base field has characteristic zero. Since the Picard group of $\M{n}$ and base point freeness do not depend on changing the base field, we may assume that the base field is $\mathbb{C}$. This allows us to utilize coinvariant divisors. Let $C = F(\left\{n-1\right\}, \left\{n\right\}, I, J)$ and $X = \left\{1, 2, \cdots, n-2\right\}$. By \cref{lem:partition}, we obtain a function $a: X \to \mathbb{N}$ such that $\sum_{i \in I} a_i = \sum_{i \in J} a_j$ and $\sum_{b \in A} a_b \neq \sum_{b \in A^c} a_b$ for $A \neq I, J$. Let $m = \sum_{i \in I} a_i + 1$ and $\mathcal{L} = \mathbb{D}_{0,n}^m(a_1, a_2, \cdots, a_{n-2}, 1, 1)$. Then, by \cref{thm:lattice} and \cref{thm:confcal},    \[ \mathcal{L}\cdot C = \mathbb{D}_{0,n}^m(a_1,a_2,\cdots, a_{n-2},1,1)\cdot F(\left\{n-1\right\}, \left\{n\right\}, I, J)=\text{deg }\mathbb{D}_{0,4}^m(1,1,m-1,m-1)=1.  \]
    Let $D\in F_{\text{Knu}}$ be another element, not equal to $C$. Suppose $D=F(\left\{n-1\right\}, \left\{n\right\}, I', J')$. Then, according to the construction, $\sum_{i\in I'}a_i\ne \sum_{j\in J'}a_j$. Without loss of generality, we may assume that $\sum_{i\in I'}a_i\ge m$. Hence, by \cref{thm:lattice},
    \[ \mathcal{L}\cdot D=\mathbb{D}_{0,n}^m\left(a_1,a_2,\cdots, a_{n-2},1,1\right)\cdot F(\left\{n-1\right\}, \left\{n\right\}, I', J')=\text{deg }\mathbb{D}_{0,4}^m\left(1,1,\sum_{i\in I'}a_i-m,\sum_{j\in J'}a_j\right).  \]
    Since $0\le \sum_{i\in I'}a_i-m,\sum_{j\in J'}a_j<m$, we can apply \cref{thm:confcal}, and deduce $\mathcal{L}\cdot D=0$.

    Now consider over an arbitrary base field. $\mathbb{D}_{0,n}^m(a_1, a_2, \cdots, a_{n-2}, 1, 1)$ is only defined over $\C$, but as mentioned in \cref{rmk:charp}, the Picard group of $\M{n}$ does not depend on the base field. Therefore, we can still define $\mathcal{L}$ as the line bundle corresponding to $\mathbb{D}_{0,n}^m(a_1, a_2, \cdots, a_{n-2}, 1, 1)$. Since the Chow ring also does not depend on the base field, as established in the preceding paragraph, $\mathcal{L}$ satisfies the intersection condition. It suffices to prove that it is base point free. By \cite[Theorem 1.2]{GG12}, $\mathcal{L}$ is a pullback of the ample line bundle on the GIT quotient $V_{d,k}\sslash_{\vec{c}}\text{SL}_{d+1}$, given that the construction of GIT quotient does not depend on the base field. Therefore, $\mathcal{L}$ is base point free.
\end{proof}

We now turn to the proof of \cref{thm:charcont} for $f_{\text{Knu}}$. 

\begin{thm}\label{thm:charKnu}
    The following sequences 
    \[  0\to \text{Pic}(\M{n-2}) \xrightarrow{(\pi_{n}^\ast, -\pi_{n-1}^\ast)}\text{Pic}(\M{n-1})\times \text{Pic}(\M{n-1})\xrightarrow{\pi_{n-1}^\ast+\pi_{n}^\ast} \text{Pic}(\M{n})\xrightarrow{c} \Z^{F_{\text{Knu}}}\to 0, \]
    and
    \[ 0\to \Z^{F_{\text{Knu}}}\xrightarrow{c^\ast}  A_1(\M{n}, \Z)\xrightarrow{(\pi_{n,\ast}, \pi_{n-1, \ast})} A_1(\M{n-1}, \Z)\times A_1(\M{n-1}, \Z)\xrightarrow{\pi_{n-1, \ast} - \pi_{n, \ast}} A_1(\M{n-2},\Z)\to 0 \]
    are exact, where $c^\ast$ realizes $F_{\text{Knu}}$ as $1$-cycles, and $c$ is the intersection with $F_{\text{Knu}}$.  In particular, \cref{thm:charKnuQ} holds for the integral Picard group.
\end{thm}

\begin{proof}
    We start by showing the exactness of the sequence 
    \[  A_1(\M{n}, \Z)\xrightarrow{(\pi_{n,\ast}, \pi_{n-1, \ast})} A_1(\M{n-1}, \Z)\times A_1(\M{n-1}, \Z)\xrightarrow{\pi_{n-1, \ast} - \pi_{n, \ast}} A_1(\M{n-2},\Z)\to 0. \]
    $\pi_{n-1, \ast} - \pi_{n, \ast}$ is surjective by \cref{thm:charproj} (this is just a corollary of the existence of a section of $\pi$). Therefore, it remains to check exactness at $A_1(\M{n-1}, \Z)\times A_1(\M{n-1}, \Z)$. 
    
    We claim that $\pi_{n-1,\ast} - \pi_{n,\ast}$ is spanned by the following two types of elements:
    \begin{enumerate}
        \item For any partition $[n-2] = I \sqcup J \sqcup K \sqcup L$ with $J$, $K$, and $L$ nonempty (while $I$ may be empty), consider the pair $(F(I \cup \{n-1\}, J, K, L), F(I \cup \{n\}, J, K, L))$.
        \item For any $F$-curve $F(\{n-1\}, J, K, L)$ on $\M{n-1}$, consider the pair $(F(\{n-1\}, J, K, L), 0)$.
    \end{enumerate}
    The $F$-curves generate $A_1(\M{n}, \mathbb{Z})$, so the quotient of $A_1(\M{n-1}, \mathbb{Z}) \times A_1(\M{n-1}, \mathbb{Z})$ by the submodule generated by elements of type (1) is naturally isomorphic to $A_1(\M{n-1}, \mathbb{Z})$. Moreover, the $F$-curves appearing in (2) are precisely those contained in $F_{n-1}$ (in the notation of \cref{thm:charproj}) on $\M{n-1}$. Therefore, by \cref{thm:charproj}, the quotient of $A_1(\M{n-1}, \mathbb{Z}) \times A_1(\M{n-1}, \mathbb{Z})$ by the submodule generated by both (1) and (2) is isomorphic to $A_1(\M{n-2}, \mathbb{Z})$. This demonstrates that the kernel of $\pi_{n-1,\ast} - \pi_{n,\ast}$ is indeed generated by the elements of types (1) and (2).
    
    Hence, it suffices to prove that (1) and (2) are included in the image of $(\pi_{n,\ast}, \pi_{n-1, \ast})$. (1) is the image of $F(I\cup\left\{n-1,n\right\},J,K,L)$ and (2) is the image of $F(\left\{n-1\right\}, J\cup\left\{n\right\},K,L)$. This completes the proof of the exactness.
    
    Let $X$ be the kernel of $(\pi_{n,\ast}, \pi_{n-1, \ast})$. Since $A_1(\M{n}, \Z)$ is a free $\Z$-module, $X$ is also a free $\Z$-module. Hence we have an exact sequence of free $\Z$-modules
    \[ 0\to X\to A_1(\M{n}, \Z)\xrightarrow{(\pi_{n,\ast}, \pi_{n-1, \ast})} A_1(\M{n-1}, \Z)\times A_1(\M{n-1}, \Z)\xrightarrow{\pi_{n-1, \ast} - \pi_{n, \ast}} A_1(\M{n-2},\Z)\to 0. \]
    Since each entry is a free $\Z$-module, the dual of this sequence is also exact. Hence, according to \cref{thm:dual}, we get an exact sequence
    \[ 0\to \text{Pic}(\M{n-2}) \xrightarrow{(\pi_{n}^\ast, -\pi_{n-1}^\ast)}\text{Pic}(\M{n-1})\times \text{Pic}(\M{n-1})\xrightarrow{\pi_{n-1}^\ast+\pi_{n}^\ast} \text{Pic}(\M{n})\to \text{Hom}(X,\Z)\to 0.\]
    Since $\text{Hom}(X,\Z)$ is also free, the cokernel of $\text{Pic}(\M{n-1})\times \text{Pic}(\M{n-1})\xrightarrow{\pi_{n-1}^\ast+\pi_{n}^\ast} \text{Pic}(\M{n})$ is torsion-free. Hence, by \cref{thm:charKnuQ} and \cref{lem:QtoZ}, the sequence
    \[  0\to \text{Pic}(\M{n-2}) \xrightarrow{(\pi_{n}^\ast, -\pi_{n-1}^\ast)}\text{Pic}(\M{n-1})\times \text{Pic}(\M{n-1})\xrightarrow{\pi_{n-1}^\ast+\pi_{n}^\ast} \text{Pic}(\M{n})\xrightarrow{c} \Z^{F_{\text{Knu}}}\to 0 \]
    is exact at $\text{Pic}(\M{n})$. The exactness of other parts except $\Z^{F_{\text{Knu}}}$ follows from the preceding argument. Moreover, \cref{cor:knusln} implies the exactness at $\Z^{F_{\text{Knu}}}$. This demonstrates the exactness of the first sequence in the statement. The exactness of the second sequence follows by taking the dual of the first sequence and applying \cref{thm:dual}.
\end{proof}

\begin{rmk}\label{rmk:Knu}
    The exact sequences in \cref{thm:charKnu} highlight a distinctive feature of Knudsen's construction. By comparing them with the exact sequences in \cref{thm:chargen}, we readily observe the structural simplicity of $f_{\text{Knu}}$—for example, $F_{\text{Knu}}$ is linearly independent. Moreover, together with \cref{cor:knusln}, this allows us to prove \cref{prop:kkne}, which shows that $\NE{f_{\text{Knu}}}$ is a simplicial cone. This simple structure of $f_{\text{Knu}}$ plays a crucial role in the proof of \cref{subsec:Mori}, where we classify certain contractions of $\M{n}$, and in \cref{thm:CDint}, which strengthens Fakhruddin's basis theorem. This represents another form of inductive structure on $f_{\text{Knu}}$, which may be useful in further applications as discussed in \cref{subsec:app}.
    
\end{rmk}

We next extend \cref{thm:chargenQ} to the integral setting. 

\begin{thm}\label{thm:chargen}
    Let $S,T\subseteq \left\{1,2,\cdots,n \right\}$ be subsets such that $|S|, |T|\ge 3$ and let $f_{S,T}$ and $F_{S,T}$ as described in \cref{subsec:cont}. 
    
    \begin{enumerate}
        \item The following sequence is exact:
        \[ 0\to \text{Pic}(\M{S\cap T}) \xrightarrow{(\pi^{S\cap T,\ast}, -\pi^{S\cap T,\ast})}\text{Pic}(\M{S})\times \text{Pic}(\M{T})\xrightarrow{\pi^{S,\ast}+\pi^{T,\ast}} \text{Pic}(\M{n})\xrightarrow{c} \Z^{F_{S,T}}.\]
        where $c$ is the intersection pairing. In particular, a line bundle $\mathcal{L}$ on $\M{n}$ is in the image of $\pi^{S,\ast}+\pi^{T,\ast}$ if and only if it intersects with curves in $F_{S,T}$ trivially.
        \item The image of  $f_{S,T}^\ast: \text{Pic}\left(\M{S}\times_{\M{S\cap T}}\M{T} \right)\to \text{Pic}\left(\M{n} \right)$ is $\text{Ker }c$.
        \item The  natural map $\text{Pic}(\M{S}) \times \text{Pic}(\M{T}) \to \text{Pic}\left(\M{S}\times_{\M{S\cap T}}\M{T}\right)$ is a surjection.
    \end{enumerate}
\end{thm}

\begin{proof}
    (1) Since $\pi^{S\cap T,\ast}$ is injective, the sequence is exact at $\text{Pic}(\M{S\cap T})$. By applying \cref{thm:charproj} repeatedly, we find that the cokernel of $\pi^{S\cap T,\ast}$ is torsion-free. Hence, the cokernel of $(\pi^{S\cap T,\ast}, -\pi^{S\cap T,\ast})$ is also torsion-free. This, together with \cref{thm:chargenQ} and \cref{lem:QtoZ}, shows that the sequence is exact at $\text{Pic}(\M{S}) \times \text{Pic}(\M{T})$.
    
    Hence, it suffices to show that the exactness at $\text{Pic}(\M{n})$. Again, by \cref{thm:chargenQ} and \cref{lem:QtoZ}, it is enough to prove that the cokernel of $\pi^{S,\ast}+\pi^{T,\ast}$ is torsion-free.

    We use induction on $n$ to prove this. If $n=4$, there is nothing to prove. Assume the induction hypothesis and consider the induction step. If $S\subseteq T$ or $T\subseteq S$, then this can be proved by iterative application of \cref{thm:charproj}. Hence, without loss of generality, we may assume that $n\not\in S$, $n-1\not\in T$. Then, the map  $\text{Pic}\left(\M{S} \right)\times \text{Pic}\left(\M{T} \right)\to \text{Pic}\left(\M{n} \right)$ factors as
    \[ \text{Pic}\left(\M{S} \right)\times \text{Pic}\left(\M{T} \right)\to \text{Pic}\left(\M{n-1} \right)\times \text{Pic}\left(\M{n-1} \right)\xrightarrow{\pi_{n}^\ast + \pi_{n-1}^\ast}  \text{Pic}\left(\M{n} \right). \]
    Note that an extension of torsion-free $\Z$-module by a torsion-free $\Z$-module is again torsion-free. By \cref{thm:charKnu}, the kernel of the map $\pi_{n}^\ast + \pi_{n-1}^\ast$ is $\text{Pic}(\M{n-2})$ and the cokernel of $\pi_{n}^\ast + \pi_{n-1}^\ast$ is torsion-free. Hence, it is enough to show that the cokernel of
    \[ h:\text{Pic}\left(\M{S} \right)\times \text{Pic}\left(\M{T} \right)\times \text{Pic}(\M{n-2})\to \text{Pic}\left(\M{n-1} \right)\times \text{Pic}\left(\M{n-1} \right) \]
    is torsion-free, where the map $\text{Pic}(\M{n-2}) \to \text{Pic}(\M{n-1})\times \text{Pic}(\M{n-1})$ is given by $(\pi_{n}^\ast, -\pi_{n-1}^\ast)$. Consider the intersection $\text{Im }h\cap (\text{Pic}(\M{n-1})\times 0)$. This is generated by $\text{Pic}(\M{S})\times 0$ and $(L,0)$, where $L\in \text{Pic}(\M{n-2})\cap \text{Pic}(\M{T})$. By the first three terms of the exact sequence of (1), which we have already proven, $\text{Pic}(\M{n-2})\cap \text{Pic}(\M{T})=\text{Pic}(\M{T\setminus\left\{n-1,n\right\}})$. Therefore, $\text{Im }h\cap (\text{Pic}(\M{n-1})\times 0)$ is generated by $\text{Pic}(\M{S})\times 0$ and $\text{Pic}(\M{T\setminus\left\{n-1,n\right\}})\times 0$. If we consider $\text{Pic}(\M{S})$ and $\text{Pic}(\M{T\setminus\left\{n-1,n\right\}})$ as submodules of $\text{Pic}(\M{n-1})$, by the induction hypothesis, the cokernel of the natural map $i:\text{Pic}(\M{S})\times \text{Pic}(\M{T\setminus\left\{n-1,n\right\}})\to \text{Pic}(\M{n-1})$ is torsion-free. Since the cokernel is also finitely generated over $\Z$, it is a free $\Z$-module. Hence, there exists a complement $X$ of the image of $i$ within $\text{Pic}(\M{n-1})$. Note that $X$ is also a free $\Z$-module. By construction, the intersection of the image of $h$ and $X\times 0$ is trivial, and the submodule of $\text{Pic}(\M{n-1})\times \text{Pic}(\M{n-1})$ generated by the image of $h$ and $X\times 0$ is $\text{Pic}(\M{n-1})\times Y$, where $Y$ is the submodule of $\text{Pic}(\M{n-1})$ generated by $\text{Pic}\left(\M{T} \right)$ and $\text{Pic}(\M{n-2})$. Again by the induction hypothesis, $Y$ has a complement $Z$, which is also a free $\Z$-module. Altogether, $X\times Z$ is a free $\Z$-submodule of $\text{Pic}\left(\M{n-1} \right)\times \text{Pic}\left(\M{n-1} \right)$, which is a complement of the image of $h$. Therefore, the cokernel of $h$ is torsion-free. This completes the proof of (1).
    
    (2) and (3) can be verified in exactly the same way as the proof of \cref{thm:charKnuQ}; therefore, we omit their proofs.
\end{proof}

We now observe that \cref{thm:charcont} for Keel's construction is simply a special case of \cref{thm:chargen} with $S = [n-1]$ and $T = \left\{1,2,3,n\right\}$.

\begin{cor}\label{cor:charkeel}
     A line bundle on $\M{n}$ is in the image of $f_{\text{Keel}}^\ast$ if and only if it intersects trivially with F-curves in $F_{\text{Keel}}$.
\end{cor}

Note that \cref{cor:charkeel} is also a special case of a theorem by V. Alexeev in \cite[Proposition 4.6]{Fak12}. Alexeev also proved a result  analogous to \cref{thm:charcont} for certain moduli of weighted pointed rational curves; see \cite[Proposition 4.6]{Fak12}.

\section{Consequences of the Main Theorem}\label{sec:cons}

\subsection{Direct Consequences}\label{subsec:cons}

In this section, we list some of the direct corollaries of the main theorems. The characterization of line bundles, as given in \cref{thm:charKap} and \cref{thm:chargen}, naturally leads to the characterization of morphisms from $\M{n}$.

\begin{cor}\label{cor:morKap}
    Let $f:\overline{\mathrm{M}}_{0,n}\to X$ be a morphism to a projective variety $X$ that contracts $\Delta_{0, \left\{i,n\right\}}$ (cf. \cref{subsec:nota}) to a point for every $1\le i\le n-1$, or equivalently, contracts the F-curves in $F_{\text{Kap}}$. Then $f$ factors through $f_{\text{Kap}}$. In particular, if $f$ is such a contraction which is not constant, then $f=f_{\text{Kap}}.$
\end{cor}
\begin{proof}
    The equivalence can also be derived from the proof of \cref{thm:charKap}. Let $\mathcal{O}_X(1)$ be an ample line bundle on $X$. Then $f^\ast \mathcal{O}_X(1)$ satisfies the assumptions of \cref{thm:charKap}, and hence it is a constant multiple of $\psi_n$. This line bundle corresponds to Kapranov’s construction, so $f$ contracts every fiber of $f_{\text{Kap}}$. The first assertion then follows directly from \cite[Lemma 1.15 (b)]{Deb01}, and the second follows from this together with the fact that $\mathbb{P}^{n-3}$ admits no nontrivial contractions.
\end{proof}

\begin{cor}\label{cor:morgen}
    Let $S,T\subseteq \left\{1,2,\cdots,n \right\}$ be subsets such that $|S|, |T|\ge 3$ and let $f_{S,T}$ and $F_{S,T}$ as described in \cref{subsec:cont}. Also, let $f:\overline{\mathrm{M}}_{0,n}\to X$ be a morphism to a projective variety $X$ that contracts the F-curves in $F_{S,T}$ to a point. Then $f$ factors through $f_{S,T}$. 
\end{cor}

\begin{proof}
    The proof is analogous to the proof of \cref{cor:morKap}, using \cref{thm:chargen} instead of \cref{thm:charKap}. 
\end{proof}

One of the obstructions of computing $\NE{f_{\text{S,T}}}$ is that its rank is typically large. For example, $\NE{f_{\text{Kap}}}$ generate a codimension $1$ subspace of $A_1(\M{n})$, so the determination of them is only slightly less difficult than proving the F-conjecture. Therefore, we could view the calculation of the relative cone of curves for such contractions as a refinement of the F-conjecture. A significant benefit of the main theorems (\cref{thm:charKap}, \cref{thm:charKnu}, and \cref{thm:chargen}) is that it characterizes morphisms $f:\M{n}\to X$ which factor through $f_{\text{Kap}}$ or $f_{\text{S,T}}$, without the need to calculate $\NE{f_{\text{Kap}}}$ or $\NE{f_{\text{S,T}}}$. 

Since \cref{thm:charcont} pertains to a line bundle, it consequently induces a corresponding dual statement about $A_1(\M{n})$. Due to the statement of \cref{thm:charKap}, \cref{thm:charKnu} and \cref{thm:chargen}, we obtain results on the relation between F-curves.

\begin{cor}\label{cor:curveKap}
    $A_1(F_{\text{Kap}})$ is a codimension $1$ subspace of $A_1(\M{n})$ and does not contain any F-curves other than those in $F_{\text{Kap}}$.
\end{cor}
\begin{proof}
    The first assertion offers an alternative interpretation of \cref{thm:charKap}. The second assertion stems from the fact that for any F-curve $F$ not contained in $F_{\text{Kap}}$, the product $\psi_n\cdot F$ equals 1.
\end{proof}

\begin{cor}\label{cor:Knuind}
    $F_{\text{Knu}}$ is a linearly independent subset of $A_1(\M{n})$.
\end{cor}

\begin{proof}
    Consider the exact sequence in \cref{thm:charKnuQ} (1). Using \cref{thm:basis} and \cref{thm:propvac}, the image of $\pi_{n-1}^\ast + \pi_{n}^\ast$ is the $\Q$-vector space generated by $\mathbb{D}_{0,n}^2(a_\bullet)$, such that $a_{n-1}$ or $a_n = 0$. Hence, the codimension of the image of $\pi_{n-1}^\ast + \pi_{n}^\ast$ inside $\text{Pic}(\M{n})_\Q$ is $2^{n-3} - 1$. This coincides with the cardinality of $F{\text{Knu}}$. Each curve in $F_{\text{Knu}}$ defines a subspace of $\text{Pic}(\M{n})_\Q$ whose codimension of the total intersection is the same as the cardinality of $F_{\text{Knu}}$, indicating that $F_{\text{Knu}}$ is linearly independent.
\end{proof}

An alternative proof also follows from \cref{cor:knusln}. This result is analogous to \cite[Proposition 4.1]{AGSS12} and \cite[Theorem 2.1 (1)]{AGS14}, but it is stronger. While these theorems deal with $\lfloor\frac{n}{2}\rfloor-1$ symmetric F-curves, \cref{cor:Knuind} concerns $2^{n-3}-1$ ordinary F-curves. In the same way as the proof of \cref{cor:Knuind}, we can apply \cref{thm:chargen} to prove a statement about F-curves. 

\cref{cor:Knuind} suggests that the relative closed cone of curves $\NE{f_{\text{Knu}}}$ might be the simplicial cone generated by $F_{\text{Knu}}$. This turns out to be true.

\begin{prop}\label{prop:kkne}
    $\NE{f_{\text{Knu}}}$ is the simplicial cone generated by $F_{\text{Knu}}$.
\end{prop}

\begin{proof}
    By \cref{thm:charKnu} and \cref{cor:Knuind}, any element of $\NE{f_{\text{Knu}}}$ is a linear combination of $F_{\text{Knu}}$. It is enough to prove that every coefficient is nonnegative. This directly follows from \cref{cor:knusln}.
\end{proof}

Note that it is hard to directly compute the cone $\NE{f_{\text{Knu}}}$, since we need to consider a limit of effective curve classes. However, the nef divisors given by coinvariant divisors in \cref{cor:knusln} allow us to work on the dual space and make the proof easier.

\cref{thm:chargenQ} allows us to compute the dimension of the space generated by $F_{S,T}$. Although the intersection pairing $c$ may not be surjective in this case—unlike in \cref{thm:charKnuQ}—the spaces $\text{Span }F_{S,T}$ and $\text{Im }c$ are dual to each other. Therefore, we can apply \cref{thm:chargenQ} (1) to obtain:

\begin{cor}\label{cor:fdimgen}
    \[ \text{dim}_\Q \text{ Span }F_{S,T}=\text{rank }\text{Pic}(\M{n})-\text{rank }\text{Pic}(\M{S})-\text{rank }\text{Pic}(\M{T})+\text{rank }\text{Pic}(\M{S\cap T}). \]
\end{cor}

Finally, using \cref{thm:charKap} and \cref{thm:chargen}, we can produce extremal rays of nef cone of $\M{n}$.

\begin{cor}\label{cor:psiext}
    $\psi_i$ generates an extremal ray of the cone of F-nef divisors of $\M{n}$. In particular, It generates an extremal ray of the nef cone of $\M{n}$.
\end{cor}

\begin{proof}
    By \cref{thm:charcont}, $F_{\text{Kap}}$ characterizes $\psi_n$ as an extremal ray. 
\end{proof}

\begin{cor}\label{cor:genext}
    Let $\pi^S:\M{n}\to \M{S}$ be the projection. If $\mathcal{L}$ generates an extremal ray of nef cone (resp. F-nef cone) of $\M{S}$, then $\pi^{S,\ast}\mathcal{L}$ also generates an extremal ray of nef cone (resp. F-nef cone) of $\M{n}$.
\end{cor}

\begin{proof}
    First, assume that $\mathcal{L}$ generates an extremal ray of the nef cone of $\M{S}$. Then, there exist curves $C_1, \ldots, C_r$ on $\M{S}$ such that if a nef line bundle intersects trivially with $C_1, \ldots, C_r$, then it is a constant multiple of $\mathcal{L}$. There exist curves $D_1, \ldots, D_r$ on $\M{n}$ such that $\pi^S(D_i) = n_iC_i$ for each $i$ and a positive integer $n_i$. Therefore, if a nef line bundle on $\M{n}$ intersects trivially with both $F_{S,T}$ and $D_1,\cdots, D_r$, where $T$ is any subset with $|T| = 3$, by \cref{thm:chargen}, it is a constant multiple of $\pi^{S,\ast}\mathcal{L}$. Note that a line bundle on $\M{S}$ is nef if and only if its pullback to $\M{n}$ is nef. The statement for the F-nef cone can be proven similarly, given that any F-curve on $\M{S}$ is the image of an F-curve on $\M{n}$.
\end{proof}

\cref{cor:psiext} and \cref{cor:genext} collectively generate a lot of extremal rays of the nef cone of $\M{n}$.

\subsection{Mori Theory of \texorpdfstring{$\M{n}$}{TEXT} with respect to \texorpdfstring{$\NE{F_\text{K\MakeLowercase{nu}}}$}{TEXT}}\label{subsec:Mori}

According to \cref{prop:kkne}, the relative cone of curves $\NE{f_{\text{Knu}}}$ from Knudsen's construction is the simplicial cone generated by $F_{\text{Knu}}$. In this section, we provide the classification of contractions of $\M{n}$ whose relative cone of curves is contained in $\NE{F_\text{Knu}}$, or equivalently, the contractions which $f_{\text{Knu}}$ factors through.

\begin{defn}\label{defn:Knucont}
For any subset $A \subseteq F_{\text{Knu}}$, we define $f_A:\M{n}\to\overline{\mathrm{M}}_{0,n}^{\text{Knu}}(A)$ to be a contraction of $\M{n}$ whose relative closed cone of curves is exactly the simplicial cone generated by $A$.
\end{defn}

The uniqueness of such a variety $\overline{\mathrm{M}}_{0,n}^{\text{Knu}}(A)$ and contraction $f_A$ follows from \cite[Proposition 1.14(b)]{Deb01}, while the existence is established as part of the following theorem.

\begin{thm}\label{thm:secknumori}
    We can associate:
    \begin{description}
        \item[Object] For any subset $A\subseteq F_{\text{Knu}}$, a projective variety $\overline{\mathrm{M}}_{0,n}^{\text{Knu}}(A)$ and a contraction  
        \[ f_A:\M{n}\to\overline{\mathrm{M}}_{0,n}^{\text{Knu}}(A) \]
        as defined in \cref{defn:Knucont}.
        \item[Morphism] For any two subsets $A,B\subseteq F_{\text{Knu}}$ such that $A\subseteq B$, a birational contraction 
        \[f_{A,B}:\overline{\mathrm{M}}_{0,n}^{\text{Knu}}(A)\to \overline{\mathrm{M}}_{0,n}^{\text{Knu}}(B).\]
    \end{description}
    These associations satisfy:
    \begin{enumerate}
        \item $\overline{\mathrm{M}}_{0,n}^{\text{Knu}}(\emptyset)=\overline{\mathrm{M}}_{0,n}$, $\overline{\mathrm{M}}_{0,n}^{\text{Knu}}(F_{\text{Knu}})=\M{n}\times_{\M{n-1}}\M{n}$.
        \item $f_{A,B}$ is transitive, i.e. $f_{B,C}\circ f_{A,B}=f_{A,C}$.
        \item $f_{\emptyset, F_{\text{Knu}}}=f_{\text{Knu}}$.
        \item $\NE{f_{\emptyset, A}}$ is the simplicial cone generated by $A$, and $f_{\emptyset, A}$ is the unique contraction of $\overline{\mathrm{M}}_{0,n}$ with this property. In particular, $f_A=f_{\emptyset, A}$.
        \item $f_{\emptyset, A}^\ast\text{Pic}\left(\overline{\mathrm{M}}_{0,n}^{\text{Knu}}(A) \right)$ is the set of line bundles on $\overline{\mathrm{M}}_{0,n}$ which intersect with curves in $A$ trivially. In particular, $\text{Pic}\left(\overline{\mathrm{M}}_{0,n}^{\text{Knu}}(A)\right)$ is a free abelian group of rank $2^{n-1}-\binom{n}{2}-|A|-1$.
    \end{enumerate}
\end{thm}

This theorem relies on \cref{cor:knusln}. 

\begin{proof}
    First, we construct $\M{n}^{\text{Knu}}(A)$. For each $C \in F_{\text{Knu}}$, let $\mathcal{L}_C$ be a line bundle that satisfies the condition stated in \cref{cor:knusln}, and let $\phi_C: \M{n} \to \mathbb{P}^{N_C}$ be the morphism corresponding to $\mathcal{L}_C$. Consider the morphism 
    \[ f_{\text{Knu}}\times\prod_{C\in A^c}\phi_C:\M{n}\to \left(\M{n-1}\times_{\M{n-2}} \M{n-1} \right)\times\prod_{C\in A^c}\mathbb{P}^{N_C}. \]
    Let $X_A^0$ be the image of this map and $f_A^0:\M{n}\to X_A^0$ be the corresponding map. By taking the Stein factorization of $f_A^0$, we obtain a contraction $\M{n}\to X_A$. Define $\M{n}^{\text{Knu}}(A):=X_A$ and let $f_{\emptyset, A}$ be the map $\M{n}\to X_A=\M{n}^{\text{Knu}}(A)$. $\M{n}^{\text{Knu}}(A)$ is automatically a projective variety by the construction. Note that an element of $A$ is contracted by $f_{\text{Knu}}$ and $\phi_C$ for every $C\in A^c$. Hence, any curve in $A$ is contracted by $f_{\emptyset, A}$. This implies that the image of $f_{\emptyset, A}^\ast:\text{Pic}\left(\M{n}^{\text{Knu}}(A)\right)\to \text{Pic}\left(\M{n}\right)$ intersects trivially with every element of $A$. Note that, by the construction, the image of $f_{\emptyset, A}^\ast$ contains the images of $f_{\text{Knu}}$ and line bundles $\mathcal{L}_C$ for $C\in A^c$. Since $\mathcal{L}_C\cdot C=1$ and $\mathcal{L}_C\cdot D=0$ for other $D\in A$, \cref{thm:charKnu} (1) implies that the image of $f_{\text{Knu}}^\ast$ is exactly the set of line bundles on $\overline{\mathrm{M}}_{0,n}$ which trivially intersect with curves in $A$. With \cref{cor:Knuind}, this completes the proof of (5).

    By definition, $f_{\text{Knu}}$ factors through $f_{\emptyset, A}$. Hence, $\NE{f_{\emptyset, A}} \subseteq \NE{f_{\text{Knu}}}$, where the latter is the simplicial cone generated by $F_{\text{Knu}}$, according to \cref{prop:kkne}, and the former is an extremal face of the latter. Therefore, $\NE{f_{\emptyset, A}}$ is generated by the curves in $F_{\text{Knu}}$ that are contracted by $f_{\emptyset, A}$. This demonstrates the first assertion of (4). The second assertion of (4) follows from \cite[Proposition 1.14(b)]{Deb01}. In particular, this shows that our construction satisfies the conditions of \cref{defn:Knucont}.

    By the uniqueness of (4) and \cref{prop:kkne}, since $f_{\text{Knu}}$ is a contraction, it follows that $\overline{\mathrm{M}}_{0,n}^{\text{Knu}}(\emptyset)=\overline{\mathrm{M}}_{0,n}$, $\overline{\mathrm{M}}_{0,n}^{\text{Knu}}(F_{\text{Knu}})=\M{n}\times_{\M{n-1}}\M{n}$ and $f_{\emptyset, F_{\text{knu}}}=f_{\text{Knu}}$. This completes the proof of (1) and (3). According to (4) and \cite[Proposition 1.14]{Deb01}, for any $A\subseteq B$, there exists a unique morphism $f_{A,B}:\M{n}^{\text{Knu}}(A)\to \M{n}^{\text{Knu}}(B)$ that satisfies $f_{\emptyset, B}=f_{A,B}\circ f_{\emptyset, A}$. Since $f_{\emptyset, F_{\text{Knu}}}=f_{\text{Knu}}$ is a birational morphism, by $f_{\emptyset, F_{\text{Knu}}}=f_{A, F_{\text{Knu}}}\circ f_{\emptyset, A}$, $f_{\emptyset, A}$ is also birational. Moreover, because $f_{\emptyset, B}=f_{A,B}\circ f_{\emptyset, A}$, $f_{A,B}$ is birational for every $A\subseteq B$. Since $f_{\emptyset, A}$ is a contraction, $\M{n}^{\text{Knu}}(A)$ is normal. Therefore, $f_{A,B}$ is a birational morphism between normal varieties, hence it is also a contraction. Since $f_{A,C}$ and $f_{B,C} \circ f_{A,B}$ are two birational morphisms between $\M{n}^{\text{Knu}}(A)$ and $\M{n}^{\text{Knu}}(C)$ that coincide on $\rm{M}_{0,n}$, they are the same morphism. This completes the proof of (2).
\end{proof}

\begin{rmk}\label{rmk:proj}
   An essential part of the proof of \cref{thm:secknumori} is showing that $\M{n}^{\text{Knu}}(A)$ is projective. As we will see later, there are alternative constructions of $\M{n}^{\text{Knu}}(A)$ that do not require \cref{cor:knusln}. However, to prove projectivity, we use the use of coinvariant divisors. In \cref{eg:segre}, we will explore why proving the projectivity of contractions of $\M{n}$ is a nontrivial task. This underscores the utility of coinvariant divisors: they offer a vast family of base-point-free line bundles on $\M{n}$, whose intersections with F-curves can be readily computed.
\end{rmk}

\begin{eg}\label{eg:segre} (Non-projective contractions of $\M{6}$)
    In \cite{Mo15}, it is demonstrated that the Segre cubic $\mathcal{S}_3$ is a contraction of $\M{6}$. Indeed, $\mathcal{S}_3$ has 10 singular points, and $\M{6}$ can be obtained by blowing up all of them. Let $\rho: \M{6} \to \mathcal{S}_3$ be the contraction map. For each singular point, the exceptional locus of the blow-up is isomorphic to $\mathbb{P}^1 \times \mathbb{P}^1$ and corresponds to the boundary divisor $\Delta_{\left\{a,b,c\right\}}$ of $\M{6}$. In other words, $\mathcal{S}_3$ can be obtained by contracting every boundary divisor of the form $\Delta_{\left\{a,b,c\right\}}$ to a point.

    Hence, $\rho$ serves as a resolution of $\mathcal{S}_3$. However, there exist resolutions of $\mathcal{S}_3$ that are smaller than $\M{6}$, as explained in \cite{Fin87}. This phenomenon can be understood in terms of F-curves. Since $\M{6}$ is a log Fano variety and any F-curve generates an extremal ray of $\NE{\M{6}}$, it is possible to contract any F-curve on $\M{6}$. Let $p$ be a singular point of $\mathcal{S}_3$ corresponding to the boundary divisor $\Delta_{\left\{1,2,3\right\}}$. This divisor contains $F(\left\{1\right\},\left\{2\right\},\left\{3\right\},\left\{4,5,6\right\})$ and $F(\left\{4\right\},\left\{5\right\},\left\{6\right\},\left\{1,2,3\right\})$. Contracting one of them yields a smooth variety, with the fiber over $p$ being $\mathbb{P}^1$. Therefore, for each singular point of $\mathcal{S}_3$, there are two small resolutions, corresponding to two F-curves contained in the respective boundary divisor. These local resolutions can be glued together to obtain a small resolution of $\mathcal{S}_3$, which still acts as a contraction of $\M{6}$. Given the choice between two options for each point, there are a total of $2^{10} = 1024$ small resolutions of $\mathcal{S}_3$. These resolutions are classified in \cite{Fin87}. There are a total of $13$ isomorphism classes of these small resolutions, and $7$ of them are non-projective. Since all of them can be realized as contractions of $\M{6}$ through gluing, they constitute a large family of non-projective contractions of $\M{6}$. This demonstrates the non-triviality of establishing the projectivity of $\M{n}^{\text{Knu}}(A)$.
\end{eg}

There are plenty of examples of nonprojective contractions of $\M{n}$ in \cite[Section 11]{MSvAX18}.

We can pose natural questions stemming from \cref{thm:secknumori}.

\begin{qes}\label{qes:modular}
    Is it possible to find a modular description of $\overline{\mathrm{M}}_{0,n}^{\text{Knu}}(A)$?
\end{qes}

Recall that $\overline{\mathrm{M}}_{0,n}^{\text{Knu}}(A)$ is constructed in terms of $\mathfrak{sl}_n$ level $1$ coinvariant divisors.These coinvariant divisors are pullbacks of line bundles from Hassett's moduli space of weighted pointed stable curves $\overline{\rm{M}}_{0, \mathcal{A}}$, according to \cite[Proposition 4.7]{Fak12} and \cite[Theorem A]{AGSS12}, we expect that the modular description of $\overline{\mathrm{M}}_{0,n}^{\text{Knu}}(A)$ is related to $\overline{\rm{M}}_{0, \mathcal{A}}$. However, the naive approach does not work: $\M{n-1}\times_{\M{n-2}}\M{n-1}$  is Hassett's moduli space corresponding to the weight $(1,1,\cdots, 1,0,0)$, so the natural guess is $(1,1,\cdots, a_1,a_2)$, but  this merely yields $\overline{\mathrm{M}}_{0,n}$. 

Another significant source of contractions of $\M{n}$ is the modular compactification of Smyth \cite{Smy13}. We refer to \cite{MSvAX18} for a detailed description of modular compactifications for rational curves. Note that, by \cref{prop:strknumori}, $\overline{\mathrm{M}}_{0,n}^{\text{Knu}}(A)$ admits at worst finite quotient singularities, so it can be a coarse moduli space of a certain Deligne-Mumford moduli stack with nontrivial stabilizers.

\begin{eg}\label{eg:triple}
    Define $\text{T}^i_{0,n}$ as the limit of the following diagram:
    \[ \begin{tikzcd}
       \M{n-1}\arrow[d, "\pi_{n-1}"]\arrow[rd, near start, "\pi_n"] & \M{n-1}\arrow[ld,  near start, "\pi_{i}"]\arrow[rd,  near start, swap, "\pi_n"] & \M{n-1}\arrow[d, "\pi_{n-1}"]\arrow[ld,  near start, "\pi_{i}"] \\
       \M{n-2} & \M{n-2} & \M{n-2}.
    \end{tikzcd}\]
    where the second row corresponds to the moduli space of stable pointed rational curve with the index set $\left\{1,2,\cdots, n\right\}\setminus\left\{n-1, n\right\}$, $\left\{1,2,\cdots, n\right\}\setminus\left\{i, n\right\}$ and $\left\{1,2,\cdots, n\right\}\setminus\left\{i,n-1\right\}$. This construction forms a `triple fiber product' of projections. In other words,
    \[ \text{T}_{0,n}^i:=\left(\overline{\mathrm{M}}_{0,n-1}\times_{\overline{\mathrm{M}}_{0,n-2}} \overline{\mathrm{M}}_{0,n-1} \right)\times_{\overline{\mathrm{M}}_{0,n-2}\times \overline{\mathrm{M}}_{0,n-2}}\overline{\mathrm{M}}_{0,n-1}. \] 
    
    In naive terms, this can be considered as the image of $\M{n}\to \M{n-1}\times \M{n-1}\times \M{n-1}$. This admits a natural map $f_{T,i}:\overline{\mathrm{M}}_{0,n}\to \text{T}^i_{0,n}$. We will demonstrate that $\text{T}^i_{0,n}$ is normal, which implies that $f_{T,i}$ is a birational contraction. Note that $\M{n-1}\times_{\M{n-2}}\M{n-1}$ is a local complete intersection, as elucidated in the proof of \cref{thm:gencont}. Given that $\M{n-2}\times \M{n-2}$ is smooth, and following the same reasoning as in the proof of \cref{thm:gencont}, $\text{T}_{0,n}^i$ is regularly embedded as a closed subscheme of $\left(\overline{\mathrm{M}}_{0,n-1}\times_{\overline{\mathrm{M}}_{0,n-2}} \overline{\mathrm{M}}_{0,n-1} \right)\times\overline{\mathrm{M}}_{0,n-1}$. Since the latter is a local complete intersection, $\text{T}_{0,n}^i$ also qualifies as a local complete intersection scheme. In particular, it is Cohen-Macaulay. According to \cref{thm:strKnu}, $f_{T,i}$ is an isomorphism on $\M{n}\setminus \rm{M}_C$ where $C=F_{T,i}:=F(\left\{1,2,\cdots, n\right\}\setminus\left\{i, n-1,n\right\},\left\{i\right\},\left\{n-1\right\},\left\{n\right\})$. Note that this is the only F-curve contracted by $f_{T,i}$. Moreover, it is evident that $f_{T,i}$ contracts $\rm{M}_C$ into a subvariety of codimension $2$ within $\text{T}_{0,n}^i$, since it factors through $\M{n-2}$. Consequently, $\text{T}_{0,n}^i$ is smooth outside a codimension $2$ locus. In particular, it is regular in codimension $1$. Hence, by Serre's criterion for normality, $\text{T}_{0,n}^i$ is normal. According to \cref{thm:secknumori}(4),  $\overline{\mathrm{M}}_{0,n}^{\text{Knu}}(F_{T,i})=\text{T}^i_{0,n}$. Given that $\text{T}^i_{0,n}$ is defined via fiber products, it admits a modular description.

    $\text{T}_{0,n}^i$ possesses another interesting property. By \cref{thm:secknumori} (5), the codimension of the image $f_{T,i}^\ast$ in $\text{Pic}(\M{n})$ is $1$. However, the codimension of the image of the map $\text{Pic}(\M{n-1})\times\text{Pic}(\M{n-1})\times\text{Pic}(\M{n-1})\xrightarrow{\pi_{i}^\ast+\pi_{n-1}^\ast+\pi_n^\ast}\text{Pic}(\M{n})$ is $2^{n-4}$. This can be established using \cref{thm:basis}. Therefore, unlike the case of the fiber product in \cref{thm:chargen} (3), there are many line bundles on the triple fiber product that do not originate from its components. Hence, the triple fiber product is intrinsically more complicated than the fiber product.
\end{eg}

\begin{qes}\label{qes:smooth}
    How singular is $\overline{\mathrm{M}}_{0,n}^{\text{Knu}}(A)$? Is it $\mathbb{Q}$-factorial? Furthermore, is $f_{A,A\cup \left\{C\right\}}$ a blow-up along a regular embedding?
\end{qes}

We remark that $\M{n}^{\text{Knu}}(A)$ is generally singular since $\M{n}^{\text{Knu}}(F_{\text{Knu}})=\M{n-1}\times_{\M{n-2}}\M{n-1}$. As observed in \cite[Section 1]{Ke92}, this singularity is étale locally isomorphic to the product of an affine space and the cone over a singular quadric surface. In particular, it is generally not $\Q$-factorial. We will completely answer this question in \cref{prop:strknumori}. Indeed, the singularities of $\M{n}^{\text{Knu}}(A)$ cannot be worse than those of $\M{n-1}\times_{\M{n-2}}\M{n-1}$.

We will answer \cref{qes:smooth} by demonstrating an alternative construction of $\overline{\mathrm{M}}_{0,n}^{\text{Knu}}(A)$. \cref{thm:strKnu}, related to Knudsen's construction, may not be the original contributions of this paper. They represent relatively elementary observations regarding $f_{\text{Knu}}$ which may already be well-known to experts in this subject. We have decided to include this not only to provide a complete proof of the properties of $\overline{\mathrm{M}}_{0,n}^{\text{Knu}}(A)$ but also to contribute to the literature on $\M{n}$.

Let $C=F(\left\{n-1\right\}, \left\{n\right\}, I, J)$ be an F-curve contracted by $f_{\text{Knu}}$. If $|I|, |J|\ne 1$, define
\[ \rm{M}_C:=\M{\left\{n-1,n,\bullet, \circ \right\}}\times\M{I+\bullet}\times \M{J+\circ}\hookrightarrow \M{n},\ \rm{N}_C:=\M{\ast+\bullet+\circ}\times\M{I+\bullet}\times \M{J+\circ}\hookrightarrow\M{n-1}. \]
Here, the map is induced by the clutching morphism. If $I=\left\{i\right\}$, define
\[ \rm{M}_C:=\M{\left\{n-1,n,i, \circ \right\}}\times \M{J+\circ}\hookrightarrow \M{n},\ \rm{N}_C:=\M{i+\ast+\circ}\times\M{J+\circ}\hookrightarrow\M{n-1}. \]
The map is also the clutching map. We can define $\rm{M}_C,\rm{N}_C$ similarly when $|J|=1$. There is a natural projection from $\rm{M}_C\to \rm{N}_C$, where both $n$ and $n-1$ map to $\ast$. The fiber of this projection is exactly the F-curves of type $C$. In other words, $\rm{N}_C$ parametrizes F-curves of type $C$, and $\mathrm{M}_C$ is the union of curves that are obviously rationally equivalent to $C$. Through the construction, if we realize $\rm{N}_C$ as a closed subvariety of $\M{n-1}\times_{\M{n-2}}\M{n-1}$ by $\rm{N}_C\hookrightarrow\M{n-1}\xrightarrow{\Delta}\M{n-1}\times_{\M{n-2}}\M{n-1}$, the composition map $\rm{M}_C\hookrightarrow \M{n}\xrightarrow{f_{\text{Knu}}}\M{n-1}\times_{\M{n-2}}\M{n-1}$ factors through $\rm{N}_C$. Note that the $\rm{N}_C$'s for various $C\in F_{\text{Knu}}$ do not intersect in $\M{n-1}\times_{\M{n-2}}\M{n-1}$. Hence, we have the following diagram:
\begin{equation}\label{eqn:diag}
\begin{tikzcd}
     \coprod_{C\in F_{\text{Knu}}}\rm{M}_C\arrow[rr, hook]\arrow[d]& & \M{n}\arrow[d, "f_{\text{Knu}}"] \\
    \coprod_{C\in F_{\text{Knu}}}\rm{N}_C\arrow[r, hook]& \M{n-1}\arrow[r, "\Delta"] & \M{n-1}\times_{\M{n-2}}\M{n-1}. \\
\end{tikzcd}
\end{equation}

\begin{prop}\label{thm:strKnu}
\cref{eqn:diag} is a fiber product diagram. Furthermore, if we let
\[ U:=\left(\M{n-1}\times_{\M{n-2}}\M{n-1}\right)\setminus\left(\coprod_{C\in F_{\text{Knu}}}\rm{N}_C \right) \]
then $f_{\text{Knu}}$ is an isomorphism on $f_{\text{Knu}}^{-1}(U)$.
\end{prop}

\begin{proof}
    We begin by proving that if $x,y\in \M{n}$ are closed points satisfying $f_{\text{Knu}}(x)=f_{\text{Knu}}(y)$, then $x,y\in \rm{M}_C$ for some $C\in F_{\text{Knu}}$. Let $D,E$ be stable pointed rational curves corresponding to $x,y$ and $p_1,\cdots, p_n$ and $q_1,\cdots, q_n$ be the marked points on them. $f_{\text{Knu}}(x)=f_{\text{Knu}}(y)$ holds if and only if $\pi_n([D])=\pi_n([E])$ and $\pi_{n-1}([D])=\pi_{n-1}([E])$. 
    
    If $p_n$ and $p_{n-1}$ are not in the same irreducible component of $D$, then we can still recover $D$ from $\pi_{n}([D])$ and $\pi_{n-1}([D])$.This is because, regardless of the number of special points on the irreducible component of $p_{n-1}$ and $p_n$, $\pi_n([D])$ retains the location of $p_{n-1}$ and $\pi_{n-1}([D])$ retains the location of $p_n$. Therefore, they must lie on the same irreducible component $X$ of $D$. If $X$ contains at least three special point excluding $p_{n-1}$ and $p_n$, then $\pi_{n}([D])$ and $\pi_{n-1}([D])$ still identifies $p_{n-1}$ and $p_n$, given sufficient number of points. Consequently, $X$ can contain at most four special points. If there is exactly one special point, $r$, excluding $p_{n-1}$ and $p_n$ on $X$, then $r$ must be a node and also contained in another irreducible component $Y$ of $D$. Since $\pi_{n-1}([D])=\pi_{n-1}([E])$ and $\pi_n([D])=\pi_n([E])$, $E$ can only be derived by attaching an $n$th point on $\pi_{n-1}([D])$ which must be placed on $Y$. If $Y$ has more than three special points, then obtaining such an $E \neq D$ in this manner is impossible, due to the sufficient number of special points. Altogether, either $X$ contains four special points or $X,Y$ both contain three special points. In any case, $D$ is contained in $\rm{M}_C$ for some $C\in F_{\text{Knu}}$. 

    Now, the second assertion is evident since $f_{\text{Knu}}$ is a proper birational map between normal varieties, which is one-to-one on $f_{\text{Knu}}^{-1}(U)$. Indeed, the Zariski Main Theorem implies that over $f_{\text{Knu}}^{-1}(U)$, it is a finite morphism, and hence, due to normality and birationality, it is an isomorphism. The first assertion is also direct, as the computation of $f_{\text{Knu}}$ on $\mathrm{M}_C$ is straightforward.
\end{proof}

This theorem elucidates the structure of $f_{\text{Knu}}$. For instance, it yields the following stronger version of \cref{prop:kkne} for $f_{\text{Knu}}$:

\begin{cor}\label{cor:intKnucont}
If an integral curve $D$ on $\M{n}$ is contracted by $f_{\text{Knu}}$, then it is rationally equivalent to a curve in $F_{\text{Knu}}$.
\end{cor}

\begin{proof}
\cref{thm:strKnu} implies that, in the context of \cref{eqn:diag}, the morphism $f_{\text{Knu}}$ contracts $\mathrm{M}_C$ precisely to $\mathrm{N}_C$ for each $C \in F_{\text{Knu}}$. Therefore, if an integral curve $D$ is contracted by $f_{\text{Knu}}$, it must be a fiber of $\mathrm{M}_C \to \mathrm{N}_C$ for some $C \in F_{\text{Knu}}$. Since $N_C$ is rational and both $C$ and $D$ are fibers of the same fibration, $D$ is rationally equivalent to $C$.
\end{proof}

One might guess that \cref{thm:strKnu} implies \cref{cor:Knuind} or \cref{thm:secknumori} since it shows that $f_{\text{Knu}}$ contracts disjoint boundary strata through projection. However, it is insufficient to imply either of them. \cref{eg:segre} demonstrates why it is not enough to prove \cref{thm:secknumori}: even though we can individually contract boundary strata, their gluing might not result in a projective variety. \cref{eg:hironaka} shows that \cref{thm:strKnu} cannot prove both of them.

\begin{eg}\label{eg:hironaka}(A variant of Hironaka's example)
    Let $C,D$ be two rational, smooth, closed subcurves of $\mathbb{P}^3$ that intersect exactly at two points $P,Q$. For example, we can let $C$ be a planar conic and $D$ be a line contained in the same plane. Let $X_0$ be the blow-up of $\mathbb{P}^3$ along $C$ and $D_0$ be the proper transform of $D$, with points $P'$ and $Q'$ corresponding to $P$ and $Q$, respectively. Let $X$ be the blow-up of $X_0$ along $D_0$ and let $F$ be the exceptional divisor. Since $F\to D_0$ is a flat family of curves, If we let $F_P$ and $F_Q$ be the fibers at $P'$ and $Q'$, we find that they are rationally equivalent subcurves of $X$.  Let $\pi:X\to \mathbb{P}^3$ be the natural projection, and let $E:=\pi^{-1}(C)$. Again, $E\to C$ is a flat family of curves. Hence, if we let $E_P$ and $E_Q$ be the fibers at $P$ and $Q$, respectively, then they are rationally equivalent subcurves. Note that $E_P$ (resp. $E_Q$) consists of two irreducible components, one of which is $F_P$ (resp. $F_Q$). Let $E_P'$ (resp. $E_Q'$) be the other component. Then $E_P=E_P'+F_P$, $E_Q=E_Q'+F_Q$. Hence, $E_P'$ and $E_Q'$ are disjoint, rationally equivalent, integral curves. 
    
    Let $Y$ be the blow-up of $\mathbb{P}^3$ along the union of $C$ and $D$. By the universal property of blow-up, there exists a natural morphism $f:X\to Y$. the map $f$ is an isomorphism on $X\setminus (E_P'\cup E_Q')$, and contracts $E_P'$ and $E_Q'$ respectively to points. Hence, the situation is exactly the same as $f_{\text{Knu}}$, but the disjoint fibers are rationally equivalent. More precisely, the rational equivalence of $E_P'$ and $E_Q'$ is given by the following: Start from $E_P'$. You first `borrow' $F_P$, so you can move through $E$. When you arrive at $Q$, you get $E_Q'$ and `win' $F_Q$, which you can `pay back' to $F$. \cref{cor:Knuind} demonstrates this method is impossible in $f_{\text{Knu}}$, hence it is not just a corollary of \cref{thm:strKnu}. In terms of projectivity, by gluing local blow-ups, we can contract only one of $E_P', E_Q'$. However, since they are rationally equivalent, such contraction cannot be projective. With \cref{eg:segre}, this demonstrates the non-triviality of \cref{thm:secknumori}.
\end{eg}

Now, we can examine the singularity of $\M{n}^{\text{Knu}}(A)$. Define $U$ as in \cref{thm:strKnu}. For $A\subseteq F_{\text{Knu}}$ let 
 \[ U_A:=\left(\M{n-1}\times_{\M{n-2}}\M{n-1}\right)\setminus\left(\coprod_{C\in A^c}\rm{N}_C \right) \text{ and }\M{n}(A):=U_A\coprod_{U}f_{\text{Knu}}^{-1}(U_{A^c}), \]
 i.e. the gluing of $U_A$ and $f_{\text{Knu}}^{-1}(U_{A^c})$ along $U$.

\begin{prop}\label{prop:strknumori}
    $\M{n}^{\text{Knu}}(A)=\M{n}(A)$. Hence, the singularity of $\M{n}^{\text{Knu}}(A)$ cannot be worse than that of $\M{n-1}\times_{\M{n-2}}\M{n-1}$. In particular, if we let
    \[F_{\text{Knu}}^{1}= \left\{F(I,J,K,L)\ |\ K=\left\{n\right\}, L=\left\{n-1\right\}, |I|\text{ or }|J|=1  \right\}\subseteq F_{\text{Knu}}, \]
    \begin{enumerate}
        \item $\overline{\mathrm{M}}_{0,n}^{\text{Knu}}(A)$ is a local complete intersection.
        \item $\overline{\mathrm{M}}_{0,n}^{\text{Knu}}(A)$ is smooth if and only if $A\subseteq F_{\text{Knu}}^{1}$. If $A\not\subseteq F_{\text{Knu}}^{1}$, then $\overline{\mathrm{M}}_{0,n}^{\text{Knu}}(A)$ is not even $\mathbb{Q}$-factorial.
        \item $f_{A, A\cup [C]}$ is a blow-up along a regular embedding (hence a divisorial contraction) if and only if $C\in F_{\text{Knu}}^{1}$. Otherwise, this is a small contraction.
    \end{enumerate}
\end{prop}

\begin{proof}
    First, we prove that the image of $\rm{M}_C$ under $f_A:=f_{\emptyset, A}$ is $\rm{M}_C$ if $C\not\in A$ and $\rm{N}_C$ if $C\in A$. By \cref{thm:strKnu}, the image of $\rm{M}_C$ under $f_{\text{Knu}}$ is $\rm{N}_C$, thus the image must be at least as large as $\rm{N}_C$. Since $\rm{M}_C\simeq\rm{N}_C\times \mathbb{P}^1$, the image is $\rm{M}_C$ if the map does not contract the $\mathbb{P}^1$ part, and $\rm{N}_C$ if the map contracts it. Since the $\mathbb{P}^1$ part of $\rm{M}_C$ corresponds to $C$, so the image is $\rm{M}_C$ if $C\not\in A$ and $\rm{N}_C$ if $C\in A$. 

    Now we prove that $f=f_{A,F_{\text{Knu}}} : \M{n}^{\text{Knu}}(A) \to \M{n-1}\times_{\M{n-2}}\M{n-1}$ is an isomorphism on $f^{-1}(U_{A})$. Since $f|_{f^{-1}(U_{A})}$ is a proper birational map between normal varieties, it is enough to show that it is one-to-one on closed points. This is confirmed by the first paragraph. Hence, we have the induced isomorphism $f^{-1}:U_A\to f^{-1}(U_{A})$.

    The restriction of $f_{A}:\M{n}\to \M{n}^{\text{Knu}}(A)$ to $f_{\text{Knu}}^{-1}(U_{A^c})$ and $f^{-1}:U_A\to f^{-1}(U_{A})$ coincides on $U$, since both maps are isomorphisms on $U$. Hence, by gluing these maps, we obtain $f':\M{n}(A)\to \M{n}^{\text{Knu}}(A)$. It is enough to show that this is an isomorphism. For the same reason as above, it is enough to show that this is a one-to-one map on closed points. Again, this follows from the first paragraph. 

    By Step 1 of \cref{thm:gencont}, $\M{n-1}\times_{\M{n-2}}\M{n-1}$ is a local complete intersection. (1) follows from this. By Step 2 of \cref{thm:gencont}, $\M{n-1}\times_{\M{n-2}}\M{n-1}$ is smooth outside $\coprod_{C\not\in F_{\text{Knu}}^1}\rm{N}_C$. Moreover, by \cite[Introduction]{Ke92}, $\M{n-1}\times_{\M{n-2}}\M{n-1}$ is singular at $\coprod_{C\not\in F_{\text{Knu}}^1}\rm{N}_C$. Therefore, by the first assertion, $\M{n}^{\text{Knu}}(A)$ is smooth if and only if $A\subseteq F_{\text{Knu}}^1$.

    By \cref{thm:strKnu}, $\M{n}\to \M{n}^{\text{Knu}}(A)$ contracts $|A\cap F_{\text{Knu}}^1|$ divisors and does not affect other divisors. Therefore,  
    \[ \text{rank}_{\Q} \text{ A}_1\left(\M{n}^{\text{Knu}}(A) \right)\ge\text{rank}_{\Q} \text{ A}_1\left(\M{n} \right)- |A\cap F_{\text{Knu}}^1| \]
    By \cref{thm:secknumori} (5), 
    \[ \text{rank}_{\Q} \text{ Pic}\left(\M{n}^{\text{Knu}}(A) \right)\otimes\Q = \text{rank}_{\Q} \text{ Pic}\left(\M{n} \right)\otimes\Q- |A| \]
    Hence, if $A\not\subseteq F_{\text{Knu}}^1$, then $\M{n}^{\text{Knu}}(A)$ is not $\Q$-factorial. This completes the proof of (2).

    (3) For $C\in F_{\text{Knu}}$, the exceptional locus of $f_{A, A\cup [C]}$ is $\rm{M}_C$, and $f_{A, A\cup [C]}\left(\rm{M}_C\right)=\rm{N}_C$. Hence this is a small contraction if $C\not\in F_{\text{Knu}}^1$. Since $\rm{N}_C$ is contained in the smooth locus of $\M{n-1}\times_{\M{n-2}}\M{n-1}$ of $C\in F_{\text{Knu}}^1$, $f_{A, A\cup [C]}$ is a blow-up along $\rm{N}_C$ in this case, hence a blow-up along a regular embedding.
\end{proof}

This gives a fairly complete description of the singularities and morphisms between $\M{n}^{\text{Knu}}(A)$. In particular, among contractions of $\M{n}$, $\overline{\mathrm{M}}_{0,n}^{\text{Knu}}\left(F_{\text{Knu}}^{1}\right)$ is the minimal resolution of $\M{n-1}\times_{\M{n-2}}\M{n-1}$.  \cref{prop:strknumori} (3) may suggest constructing $\M{n}^{\text{Knu}}(A)$ via blow up of closed subschemes of $\M{n-1}\times_{\M{n-2}}\M{n-1}$. However, this seems challenging, since the construction of $\M{n}$ from $\M{n-1}\times_{\M{n-2}}\M{n-1}$ is not local. Recall that, in \cite[Definition 2.3]{Knu83}, Knudsen defined a coherent sheaf $\mathcal{K}$ on $X=\M{n-1}\times_{\M{n-2}}\M{n-1}$ by
\[ 0\to \mathcal{O}_X\to \mathcal{O}_X(\Delta)\times \mathcal{O}_X(s_1+\cdots+s_{n-2})\to \mathcal{K}\to 0\]
and proved $\text{Proj}(\text{Sym }\mathcal{K})=\M{n}$. Hence, we cannot distinguish $\rm{N}_C$'s in this construction. \cref{cor:knusln} plays a key role in constructing line bundles that distinguish $\rm{N}_C$'s and to construct $\M{n}^{\text{Knu}}(A)$ using them.

\section{Further Discussion}\label{sec:conclu}

\subsection{Advantage of coinvariant divisors}\label{subsec:nec}

Given that the main results of this paper are proved using coinvariant divisors, it is natural to ask what advantages this approach has over alternative methods. The essential part of the proofs of the main theorems involves solving a linear equation. In principle, we can solve them directly using a classical basis such as the one provided in \cite[Lemma 2]{GF03}. However, the corresponding linear algebra problem is highly nontrivial. This complexity originates from the instability of the basis under pullback. This instability introduces various relations between line bundles on $\M{n}$ into our linear equation. Therefore, although it is theoretically possible to prove these main theorems with a classical basis, it would be an enormous task if $n$ is large. The advantage of the basis given by coinvariant divisors is that the corresponding linear equations are reasonable, as we saw in \cref{lem:linalg} and the proof of \cref{thm:charKap}. This leads to relatively simple proofs of the main theorems, written in \cref{sec:char}.

For $f_{\text{Kap}}$ and $f_{\text{Keel}}$, an alternative method exists to prove the main theorem. Moreover, if the characteristic of the base field is $0$, \cref{thm:charcont} for $f_{\text{Knu}}$, which is weaker than \cref{thm:charKnu}, can be proved without using the theory of coinvariant divisors. First we will discuss the case of $f_{\text{Kap}}$ and $f_{\text{Keel}}$. We outline the proof, communicated to the author by N. Fakhruddin: Both Kapranov's and Keel's constructions are special cases of Hassett's moduli space of weighted pointed rational curves, by \cite[Section 6]{Has03}. Moreover, their weights are contained in $(0,1]$, unlike Knudsen's construction, so \cite[Proposition 4.6]{Fak12} applies to this case.

Despite this proof, we still believe that the proof using coinvariant divisors is meaningful. This proof only applies to iterative smooth blowdowns corresponds to an F-curve, while the coinvariant divisor method can be applied to a more general set of contractions, as described by \cref{thm:charcont} and the discussion in \cref{sec:intro}. In particular, the coinvariant divisor method is also applicable to non-birational contractions. Moreover, the coinvariant divisor method is of a very algebraic nature, providing an alternative perspective to the geometric proof outlined above.

We note as further evidence that the method using coinvariant divisors is generalizable. Using the theory of conformal blocks associated to vertex operator algebras \cite{DGT21, DGT22a, DGT22b}, in \cite{Cho25}, the author generalizes the result of this paper to $\overline{\mathrm{M}}_{1,n}$. The proof relies on an analogue of \cref{thm:basis} in genus $1$, involving coinvariant divisors arising from the discrete series Virasoro VOA $\mathrm{Vir}_{5,2}$. This suggests that coinvariant divisors have the potential to be applied in broader contexts.

If the characteristic of the base field is $0$, \cite[Lemma 3.2.5]{KMM87} indeed applies to $f_{\text{Knu}}$. Let $D=\sum_{T} \Delta_{T}$ where $T$ runs over subsets of $\left\{1,2,\cdots, n-2\right\}$ such that $|T|, |\left\{1,2,\cdots, n-2\right\}\setminus T|\ge 2$. Then $K_{\M{n}}+\epsilon D$ has negative intersection with every curve in $F_{\text{Knu}}$ for $\epsilon\in \Q_{>0}$. Therefore, by \cite[Theorem 3.2.1]{KMM87}, \cref{thm:charcont} holds for $f_{\text{Knu}}$. However, this method has three significant limitations. First, it is not at all helpful for actually computing the Picard group. We still need to solve linear equations to explicitly characterize line bundles from Knudsen's construction, which is the aim of \cref{thm:charKnu}. Second, it only applies when the characteristic is zero, as in positive characteristic, the contraction theorem is only conjectured to hold. Finally, this requires \cref{prop:kkne} as an input, which is nontrivial to prove. The strategy of \cref{rmk:charp} does not apply due to the first defect, i.e., the lack of an explicit description of the Picard group. Using this method, we can also prove \cref{thm:Knumori}, by using \cite[Theorem 3.2.1]{KMM87}. However, to prove this theorem, we need not only the characteristic zero condition but also \cref{prop:kkne} and \cref{thm:Knuind} as an input, which requires the use of coinvariant divisors. 

We further note that, there is a proof of \cref{thm:Knumori} that does not utilize the coinvariant divisors, as suggested in the proof of \cref{cor:knusln}. Indeed, we can construct the line bundles in \cref{cor:knusln} via the pullback of ample divisors on the GIT quotient $V_{d,k}\sslash_{\vec{c}}\text{SL}_{d+1}$. Using \cite[Theorem 2.1]{GG12} with a set of numbers provided by \cref{lem:partition} gives the desired base point free line bundles.

Nevertheless, the method of using coinvariant divisors has the advantage of being readily generalizable. We remark that constructing a new contraction of $\M{n}$ as a GIT quotient is a nontrivial task, and studying its interaction with $F$-curves is also laborious. To the best of the author's knowledge, only those coinvariant divisors associated with $\mathfrak{sl}_n$ at level $1$, or certain special cases of $\mathfrak{sl}_2$, can be constructed via geometric invariant theory \cite{Fak12,Gia13,GG12,AGS14}. In contrast, by \cref{thm:glogen}, all coinvariant divisors are base point free, and the computation of their intersection numbers with $F$-curves is relatively easy, thanks to \cref{thm:factor} and \cref{thm:propvac}. This leaves ample room for the application of coinvariant divisors to the study of the geometry of $\M{n}$.

\subsection{Applications}\label{subsec:app}

Our study of Knudsen's construction is partially motivated by the desire to develop an inductive method, similar to the inductive approach in \cite{Ke92}, for proving theorems on $\M{n}$. The induction based on Keel's construction has the advantage of consisting of iterative blow-ups of regularly embedded subschemes. Knudsen's construction lacks this feature, but it has a simpler structure, as described in \cref{thm:charKnu} and \cref{thm:secknumori}. We present an example theorem that can be proved using such an inductive method, which is an improvement on \cref{thm:basis}.

\begin{thm}\label{thm:CDint}
    \begin{enumerate}
        \item $\text{Pic}(\M{n})$ is generated by coinvariant divisors of type A and level 1.
        \item Coinvariant divisors of $\mathfrak{sl}_2$ and level 1 form a basis of $\text{Pic}(\M{n})\otimes \Z\left[\frac{1}{2}\right]$.
    \end{enumerate}
\end{thm}

\begin{proof}
    Let $X_n$ (resp. $X_n^2$) be the submodule of $\text{Pic}(\M{n})$ generated by type A (resp. $\mathfrak{sl}_2$) level $1$ coinvariant divisors.
    
    (1) The case $n=4$ follows from \cref{thm:confcal}. The induction step is implied by the following commutative diagram:
    \[\begin{tikzcd}
        X_{n-1}\times X_{n-1}\arrow[r, "\pi_{n-1}^\ast+\pi_{n}^\ast"]\arrow[d] & X_n\arrow[r, "c"]\arrow[d]& \Z^{F_{\text{Knu}}}\arrow[r]\arrow[d, "\text{id}"] & 0\\
        \text{Pic}(\M{n-1})\times \text{Pic}(\M{n-1})\arrow[r, "\pi_{n-1}^\ast+\pi_{n}^\ast"] & \text{Pic}(\M{n}) \arrow[r, "c"] & \Z^{F_{\text{Knu}}} \arrow[r] & 0
    \end{tikzcd}\]
      where every vertical map is an inclusion. Moreover, \cref{cor:knusln} implies that both $X_n\xrightarrow{c} \Z^{F_{\text{Knu}}}$ and $\text{Pic}(\M{n})\xrightarrow{c} \Z^{F_{\text{Knu}}}$ are surjective. Moreover, the first vertical map is also surjective by the induction hypothesis. Therefore, by a straightforward diagram chasing, $X_n\to \text{Pic}(\M{n})$ is surjective .

    (2) We use induction on $n$. The case $n=4$ also follows from \cref{thm:confcal}. Assume that the theorem holds for $n-1$ and consider the case for $n$. We have the following commutative diagram    
    \[\begin{tikzcd}
        X_{n-1}^2\times X_{n-1}^2\arrow[r, "\pi_{n-1}^\ast+\pi_{n}^\ast"]\arrow[d] & X_n^2\arrow[r, "c"]\arrow[d]& I\arrow[r]\arrow[d] & 0\\
        \text{Pic}(\M{n-1})\times \text{Pic}(\M{n-1})\arrow[r, "\pi_{n-1}^\ast+\pi_{n}^\ast"] & \text{Pic}(\M{n}) \arrow[r, "c"] & \Z^{F_{\text{Knu}}} \arrow[r] & 0
    \end{tikzcd}\]
    where $I$ is the image of $X_n^2\to \text{Pic}(\M{n})\xrightarrow{c} \Z^{F_{\text{Knu}}}$. This is well-defined for the same reason as (1). By the induction hypothesis, the cokernel of $X_{n-1}^2\times X_{n-1}^2\to \text{Pic}(\M{n-1})\times \text{Pic}(\M{n-1})$ is a $2$-group. Thus, it suffices to prove that the cokernel of $I\to \Z^{F_{\text{Knu}}}$ is a $2$-group.

    To prove this, it suffices to show that the sequence
    \[ (X_{n-1}^2\times X_{n-1}^2)\otimes \Z/p\Z\to X_n^2\otimes \Z/p\Z\to (\Z/p\Z)^{F_{\text{Knu}}} \to 0 \]
    is exact, which represents the `mod $p$' version of the first row, for every odd prime number $p$. According to \cref{thm:propvac} and \cref{thm:basis}, the cokernel of $(X_{n-1}^2\times X_{n-1}^2)\otimes \Z/p\Z\to X_n^2\otimes \Z/p\Z$ is a $2^{n-3}$-dimensional $\Z/p\Z$-vector space. Hence, it is enough to prove the exactness at $X_n^2\otimes \Z/p\Z$. This can be established as in the proof of \cref{lem:Qknu}, using \cref{lem:linalg} for $k=\Z/p\Z$.
\end{proof}

 This proof reveals that the exact sequence of \cref{thm:charKnu} can be utilized in the form of an induction argument. Especially, the surjectivity of $c$ and the injectivity of $c^\ast$ explained in \cref{thm:charKnu} are crucial ingredients of the proof. This is the distinguished feature of the exact sequence associated with Knudsen's construction. For example, we have a similar exact sequence for projection maps in \cref{thm:charproj}, but in this case $\text{Ker }c^\ast$ is hard to explicitly describe, which makes it difficult to apply to the induction as in the proof of \cref{thm:CDint}. Given the effectiveness of \cref{thm:charKnu} in establishing \cref{thm:CDint}, we expect that this method can be applied to various circumstances regarding $\M{n}$.

\printbibliography
\end{document}